\documentclass[10pt,reqno]{amsart}
\usepackage{amssymb,mathrsfs,color}
\usepackage{pinlabel}
\usepackage{graphicx}

\definecolor{deepgreen}{cmyk}{1,0,1,0.5}

\newcommand{\A}{\mathcal{A}}

\newcommand{\calD}{\mathcal{D}}
\newcommand{\E}{\mathcal{E}}

\newcommand{\LL}{\mathcal{L}}

\newcommand{\HH}{\mathcal{H}}

\def\calK{\mathcal{K}}
\def\f{\frac}

\newcommand{\C}{\mathbb{C}}
\newcommand{\N}{\mathbb{N}}
\newcommand{\R}{\mathbb{R}}
\newcommand{\Z}{\mathbb{Z}}

\newcommand{\al}{\alpha}
\newcommand{\be}{\beta}
\newcommand{\ga}{\gamma}
\newcommand{\de}{\delta}
\newcommand{\e}{\varepsilon}
\newcommand{\fy}{\varphi}

\newcommand{\la}{\lambda}

\newcommand{\z}{\zeta}

\newcommand{\Om}{\Omega}

\newcommand{\La}{\Lambda}
\newcommand{\Sig}{\Sigma}

\newcommand{\p}{\partial}
\newcommand{\na}{\nabla}

\newcommand{\supp}{\operatorname{supp}}

\newcommand{\lec}{\lesssim}
\newcommand{\gec}{\gtrsim}

\newcommand{\I}{\infty}

\newcommand{\ti}{\widetilde}

\newcommand{\LR}[1]{{\langle #1 \rangle}}

\newcommand{\ds}{\displaystyle}
\newcommand{\ang}[1]{\left\langle{#1}\right\rangle}
\newcommand{\abs}[1]{\left\lvert{#1}\right\rvert}

\newcommand{\EQ}[1]{\begin{equation}\begin{split} #1 \end{split}\end{equation}}
\newcommand{\Del}[1]{}

\newcommand{\mat}[1]{\begin{pmatrix} #1 \end{pmatrix}}
\newcommand{\pt}{&}
\newcommand{\pr}{\\ &}
\newcommand{\pq}{\quad}
\newcommand{\pn}{}

\def\ti{\tilde}

\newtheorem{thm}{Theorem}[section]

\newtheorem{lem}[thm]{Lemma}
\newtheorem{prop}[thm]{Proposition}

\theoremstyle{remark}

\def\eps{\varepsilon}
\def\nn{\nonumber}
\def\weakto{\rightharpoonup}
\def\embed{\hookrightarrow}

\def\const{\mathrm{const}}
\def\lan{\langle}
\def\ran{\rangle}
\def\wt{\widetilde}
\def\glei{\mathrm{eq}}
\renewcommand\Re{\mathrm{Re}\,}
\renewcommand\Im{\mathrm{Im}\,}
\def\ol{\overline}

\newcounter{parts}

\begin{document}

\author{A.~Lawrie}
 
\author{W.~Schlag}
\address{Department of Mathematics, The University of Chicago, 5734 South University Avenue, Chicago, IL 60615, U.S.A.}
\email{schlag@math.uchicago.edu, alawrie@math.uchicago.edu}

\thanks{Support of the National Science Foundation,  DMS-0617854 (WS)  is gratefully acknowledged. The second author thanks Piotr Bizo\'n for bringing this problem to his attention. The authors thank Fritz Gesztesy for pointing out references~\cite{CHS}, \cite{Newton}. The first author thanks Brent Werness for several helpful discussions. }

\title{Scattering for wave maps exterior to a ball}

\begin{abstract}
We consider $1$-equivariant wave maps from $\R_t\times (\R^3_x\setminus B) \to S^3$ where $B$ is a ball centered at $0$, and $\p B$ gets mapped to a fixed point on~$S^3$.  We show that $1$-equivariant
maps of degree zero scatter to zero irrespective of their energy.  For positive degrees, we prove asymptotic stability of the unique harmonic maps in 
the energy class determined by the degree. 
\end{abstract}

\maketitle

\section{Introduction}

Wave maps, also known as nonlinear $\sigma$-models, are a well-studied area in physics and mathematics. They constitute a class of nonlinear wave equations
defined as critical points (at least formally) of Lagrangians
\[
\LL(u,\p_t u) = \int_{\R^{d+1}} \frac12 \big(-|\p_t u|_g^2 + \sum_{j=1}^d |\p_j u|_g^2\big)\, dt dx
\]
where $u:\R^{d+1}\to M$ is a smooth map into a Riemannian manifold $(M,g)$.  If $M\hookrightarrow \R^N$ is embedded, then critical points are characterized
by the property that $\Box u \perp T_uM$ where $\Box$ is the d'Alembertian. In particular, harmonic maps from $\R^d\to M$ are wave maps which do not depend on time. 
For a recent review of some of the main developments in the area we refer to Krieger's survey~\cite{Krieger}.  

In the presence of symmetries, such as when the target manifold~$M$ is rotationally symmetric, one often singles out a special class of such maps called equivariant wave maps. 
For example, for the sphere $M=S^{d}$ one requires that $u\circ \rho=\rho^\ell \circ u$ where $\ell$ is a positive integer and 
$\rho\in SO(d)$ acts on both $\R^d$ and $S^d$ by rotation, in the latter case about a fixed axis. 
These maps themselves have been extensively studied, see for example Shatah~\cite{Sha}, Christodoulou, Tahvildar-Zadeh~\cite{ChT}, Shatah, Tahvildar-Zadeh~\cite{ShT}.
For a summary of these developments, see   the book  Shatah, Struwe~\cite{ShSt}. 

In this paper, we investigate equivariant wave maps from $3+1$-dimensional Minkowski space exterior to a ball and with $S^3$ as target. To be specific, 
let $B\subset \R^3$ be the unit ball in~$\R^3$. 
We then consider wave maps $U: \R \times (\R^3\setminus B)\to S^3$ with a Dirichlet condition on $\partial B$, i.e., $U(\partial B)=\{N\}$
where $N$ is a fixed point on $S^3$. In the usual equivariant formulation of this equation, where $\psi$ is the azimuth angle measured from the north
pole, the equation for the $\ell$-equivariant wave map from $\R^{3+1}\to S^3$ reduces to 
\EQ{\label{WMell}
\psi_{tt} - \psi_{rr} - \frac{2}{r}\psi_r +\ell(\ell+1)\frac{\sin(2\psi)}{2r^2}=0
} 
We restrict to $\ell=1$ and $r\ge1$ with Dirichlet boundary condition $\psi(1,t)=0$ for all $t\ge0$. In other words, we
are considering the Cauchy problem
\EQ{\label{WM}
\psi_{tt} - \psi_{rr} - \frac{2}{r}\psi_r + \frac{\sin(2\psi)}{r^2}&=0, \quad r\ge1,\\ \psi(1,t)&=0,\quad \forall\:t\ge0,\\
\psi(r,0)&=\psi_0(r),\\ \psi_t(r,0)&=\psi_1(r)
}
The conserved energy is
\EQ{\label{ener}
\E(\psi,\psi_t)=\int_1^\infty \frac12\big( \psi_t^2+\psi_r^2+2\frac{\sin^2(\psi)}{r^2}\big) r^2\, dr 
}
Any $\psi(r,t)$  of finite energy  and continuous dependence on $t\in I:= (t_{0},t_{1})$ must satisfy $\psi(\infty,t)=n\pi$ for all $t\in I$ where $n\ge0$ is fixed. 

The natural space to place the solution into for $n=0$ is
the {\em energy space} $\HH:= (\dot H^1_{0}\times L^2)((1,\infty))$ with norm 
\EQ{\label{eq:HH norm}
\| (\psi,\dot \psi)\|_{\HH}^2 := \int_1^\infty (\psi_r^2(r) + \dot \psi^2(r))\, r^2\, dr  
}
Here $\dot H^1_{0}((1,\infty))$ is the completion of the smooth functions on $(1,\infty)$ with compact support under the first norm on the right-hand
side of~\eqref{eq:HH norm}.   

The exterior equation \eqref{WM} was proposed by Bizon, Chmaj, and Maliborski~\cite{BCM}  as a model in which to study the problem of relaxation to the ground
states given by the various equivariant harmonic maps. In the physics literature, this model was introduced in~\cite{Physics} as an easier alternative to the Skyrmion equation.
Moreover, \cite{Physics} stresses the  
analogy with the damped pendulum which plays an important role in our analysis. 
Numerical simulations described in~\cite{BCM} indicate that in each equivariance
class and topological class given by the boundary value~$n\pi$ at  $r=\infty$ {\em every solution} scatters to the unique harmonic map 
that lies in this class. In this paper we verify  this conjecture for  $\ell=1, n=0$. These solutions start at the north-pole and eventually return there. 
For $n\ge1$ we only obtain a  perturbative result. 

\begin{thm}\label{main}
Consider the topological class defined by equivariance $\ell=1$ and degree $n=0$. Then 
for any smooth energy data in that class there exists a unique global and smooth evolution to~\eqref{WM} which scatters to zero in the  sense
that the energy of the wave map on an arbitrary but fixed compact region vanishes as $t\to\I$.   \end{thm} 

The scattering  property can also be phrased in the following fashion: one has 
\EQ{\label{psi scat}
(\psi,\psi_t)(t)=(\fy,\fy_t)(t)+o_{\HH}(1)\quad t\to\infty
}
where $(\fy,\fy_t)\in\HH$ solves the linearized version of~\eqref{WM}, i.e.,
\EQ{\label{fy eq}
\fy_{tt} - \fy_{rr} - \frac{2}{r}\fy_r + \frac{2\fy}{r^2}=0, \; r\ge1, \:  \fy(1,t)=0
}
We prove Theorem~\ref{main} by means of the Kenig-Merle method~\cite{KM1}, \cite{KM2}. The most novel aspect of our implementation
of this method lies with the rigidity argument. Indeed, in order to prove Theorem~\ref{main} {\em without any upper bound on the energy}  
we demonstrate that the natural virial functional is globally coercive on~$\HH$. This requires a detailed variational argument, the most delicate
part of which consists of a phase-space analysis of the Euler-Lagrange equation.

The advantage of this model lies with the fact that removing the unit ball eliminates the scaling symmetry
and also renders the equation subcritical relative to the energy. Both of these features are in stark contrast to the same
equation on $3+1$-dimensional Minkowski space, which is known to be super-critical and to develop singularities in finite time, see Shatah~\cite{Sha} and also Shatah, Struwe~\cite{ShSt}.

Another striking feature of this model, which fails for the $2+1$-dimensional analogue,  lies with the fact that it admits  infinitely many stationary solutions~$Q_{n}(r)$ which
satisfy $Q_{n}(1)=0$ and $\lim_{r\to\infty}Q_{n}(r)=n\pi$, for each $n\ge1$. These solutions have minimal energy in the class of all functions of finite
energy which satisfy the~$n\pi$ boundary condition at~$r=\infty$, and they are the unique stationary solutions in that class. We denote the latter class by~$\HH_{n}$.

\begin{thm}\label{thm:2}
For any $n\ge1$ there exists $\eps>0$ small with the property that for any smooth data $(\psi_{0},\psi_{1})\in\HH_{n}$ such that 
\[
\| (\psi_{0},\psi_{1}) - (Q_{n},0)\|_{\HH}<\eps
\]
the solution to~\eqref{WMell} with data $(\psi_{0},\psi_{1})$ exists globally, is smooth,  and scatters to $(Q_{n},0)$ as $t\to\infty$. 
\end{thm}

The same result applies as well to higher equivariance
classes $\ell\ge2$, after some fairly obvious modifications of the arguments in Section~\ref{sec:boring}. However, for the sake
of simplicity we restrict ourselves to $\ell=1$.  
Scattering here means that  on compact regions in space one has $(\psi,\psi_t)(t) - (Q_{n},0)\to (0,0)$ in the energy topology, or  alternatively
\EQ{\label{psi n scat}
(\psi,\psi_t)(t)=(Q_{n},0)+(\fy,\fy_t)(t)+o_{\HH}(1)\quad t\to\infty
}
where $\fy$ solves \eqref{fy eq}.  Bizo\'n, Chmaj, and Maliborski~\cite{BCM} conducted numerical experiments
which suggest that Theorem~\ref{thm:2} should hold with $\eps=\infty$. Not only does this conjecture appear out of reach, but even the non-perturbative
regime $\eps\simeq1$ seems inaccessible, at least by the methods of this paper. The main difficulty with the implementation of the Kenig-Merle method lies with the coercivity of the virial functional centered at the harmonic maps $Q_n$. 
Indeed, in Section~\ref{The rigidity argument}, we establish the global coercivity of the virial functional centered at zero. This hinges crucially on the fact that the Euler-Lagrange equation of the associated variational problem can be transformed into an autonomous system in the plane which we analyze by a rigorous study of the phase portrait. For the nonzero $Q_n$ we lose this reduction to an autonomous system, making any rigorous statement about the Euler-Lagrange equation associated to the virial functional centered at $Q_n$ very difficult. Furthermore, no explicit expression is known for the $Q_n$ which makes even the perturbative analysis --- in and  of itself useless for the Kenig-Merle method --- of this virial functional very non-obvious. It therefore seems that the case $n\ge1$ requires a different strategy from the one we employ here.

\section{Basic well-posedness and scattering}

One has the following version of Hardy's inequality in $\dot H^1(1,\infty)$: 
\EQ{
\int_1^\infty \psi^2(r) \, dr \le 4 \int_1^\infty \psi_r^2(r)r^2\, dr
}
proved by integration by parts: 
\EQ{
\int_1^\infty \psi^2(r)\, dr + \psi^2(1)= -2\int_1^\infty  r\psi_r(r)\psi(r)\, dr 
}
and an application of Cauchy-Schwarz. This shows in particular that $\E(\vec\psi)\simeq  \|\vec\psi\|_{\HH}^2$
where $\vec\psi=(\psi,\dot\psi)$.   Another useful fact is the Strauss estimate:
\EQ{\label{strauss}
|\psi(r)|\le 2r^{-\frac12} \|\psi\|_{\dot H^1(1,\infty)}\quad\forall r\ge1
}
which in particular implies that $\|\psi\|_\infty\le 2\|\psi\|_{\dot H^1}$.  Since the nonlinearity in~\eqref{WM} is globally Lipschitz due to $r\ge1$,  
energy estimates immediately imply the following global well-posedness result. In what follows, $\R^d_*:=\R^d\setminus B$ where
$B$ is the unit ball at the origin. 

\begin{prop}\label{prop:1}
For any $(\psi_0,\psi_1)\in \HH$ the Cauchy problem \eqref{WM} has a unique global solution 
\EQ{
\psi\in C([0,\infty);\dot H^1_{0}(1,\infty)),\; \psi_t \in C([0,\infty), L^2(1,\infty))
}
in the Duhamel sense
which depends continuously on the data. Moreover, $\E(\vec \psi(t))=\const$ and we have persistence of regularity. 
\end{prop}
\begin{proof}
Just write the equation in Duhamel form and apply the standard energy estimate to obtain local well-posedness.
To be more precise, we write 
\EQ{
\vec \psi(t)& = S_0(t) \vec\psi(0) + \int_0^t S_0(t-s) (0, N(\psi))(s) \, ds, \\ N(\psi)(t,r)&:=-\frac{\sin(2\psi(t,r))}{r^2}
}
where $S_0(t)$ is the linear evolution of the wave equation in $\R^{1}_t\times \R^3_*$, with a Dirichlet condition at $r=1$ (everything can be taken to be radial, of course). 
By the conservation of energy one has 
\EQ{
\| S_0(t)\vec\psi(0)\|_{\HH} = \|\vec \psi(0)\|_{\HH}
}
whence
\EQ{
\| \vec\psi(t)\|_{\HH} &\lec \|\vec\psi(0)\|_{\HH} + \int_0^t \|\psi(s)\|_2\, ds \\
&\lec \|\vec\psi(0)\|_{\HH} + t \sup_{0<s<t}  \|\psi(s)\|_2 
}
So we can set up a contraction in the space
$
L^\infty_t(I;\HH)
$
where $I=[0,T)$ and $T$ is small depending only on the size of $\|\vec\psi(0)\|_{\HH}$. 
The global statement therefore follows by energy conservation. 
\end{proof}

As in~\cite{ShSt} we refer to these energy Duhamel solutions as {\em strong solutions}. 
For the scattering problem the formulation~\eqref{WM} is less convenient due to the linear term in the nonlinearity:
\EQ{
\frac{\sin(2\psi)}{r^2} = \frac{2\psi}{r^2} + \frac{\sin(2\psi)-2\psi}{r^2} = \frac{2\psi}{r^2} + \frac{O(\psi^3)}{r^2}
}
The presence of the strong repulsive potential $\frac{2 }{r^2}$ indicates that the linearized operator of~\eqref{WM} has more dispersion
than the three-dimensional wave equation. In fact, it has the same dispersion as the five-dimensional wave equation as the following 
standard reduction shows. 

We set $\psi=r u$ which leads to the equation
\EQ{\label{ueq}
u_{tt} - u_{rr} -\frac{4}{r} u_r + \frac{\sin(2ru)-2ru}{r^3}=0, \quad r\ge1,\; u(1,t)=0
}
The nonlinearity is of the form $N(u,r):=u^3\, Z(ru)$ where $Z$ is a smooth function, and the linear part is the d'Alembertian in $\R^1_t\times \R^5_*$. 

To relate strong solutions of~\eqref{WM} with those of~\eqref{ueq} we first note that  
\EQ{
\int_1^\infty \psi_r^2(r) r^2\, dr \simeq \int_1^\infty u_r^2(r) r^4\, dr 
}
via Hardy's inequality and the relations
\[
\psi_r = r u_r + u = r u_r + \frac{\psi}{r}
\]
Therefore, the map $\HH \ni \vec \psi  \to \frac{1}{r}\vec \psi=: \vec u \in  \dot H^{1}_0 \times L^2 ( \R^5_*)$ is an isomorphism and in what follows we will use the notation $\HH$ for both spaces without further comment. Second, there is the following Strauss estimate in $\R^5_*$:  
\EQ{
|u(r)|\lec r^{-\frac32} \|u\|_{\dot H^1} 
}

\begin{prop}
\label{prop:2}
The exterior Cauchy problem for~\eqref{ueq} is globally well-posed in $\dot H^1_{0}\times L^2(\R^5_*)$. 
Moreover, a solution $u$ scatters as $t\to\infty$ to a free wave, i.e., a solution $\vec v\in \HH$ of  
\EQ{\label{v free}
\Box v =0, \; r\ge1, \;  v(1,t)=0, \; \forall t\ge0
}
if and only if $\|u\|_S<\infty$ where $S=L^3_t([0,\infty); L^6_x(\R^5_*))$. In particular, there exists a constant $\delta>0$ small 
so that if $\|\vec u(0)\|_{\HH}<\delta$, then $u$ scatters to free waves as $t \to\pm\infty$. 
\end{prop}
\begin{proof}
By the global Strichartz\footnote{Due to the radial assumption and the simple geometry, one does not need to
resort to the sophisticated construction in~\cite{SmSo}. Indeed, grazing and gliding rays cannot occur in this setting which
is the main difficulty in the general case and which is addressed by means of the Melrose-Taylor parametrix in~\cite{SmSo}.
For the radial problem outside the ball one can instead rely on an elementary and explicit parametrix.} 
 estimates of Smith-Sogge~\cite{SmSo} for the free wave equation outside a convex obstacle every energy solution
of~\eqref{v free} satisfies 
\EQ{\label{36strich}
\| v\|_{L^3_t(\R; \dot W_x^{\frac12, 3}(\R^5_*))} \lec  \|\vec v(0)\|_{\HH} 
}
We claim the embedding $\dot W_x^{\frac12, 3} \embed L^6_x$ for radial functions in $r\ge1$ in $\R^5_*$. Indeed, one checks 
via the fundamental theorem of calculus that $\dot W_x^{ 1, 3} \embed L^\infty_x$. More precisely, 
\EQ{
|f(r)|\le r^{-\frac23} \| f\|_{\dot W_x^{1, 3} }
}
Interpolating this with the  embedding $L^3\embed L^3$ we obtain the claim. 
From~\eqref{36strich} we infer the weaker Strichartz estimate
\EQ{
\| v\|_{L^3_t(\R; L^6_x(\R^5_*))} \lec  \|\vec v(0)\|_{\HH}
}
 which suffices for our purposes. Indeed, applying it to the equation 
\[
\Box u  = u^3 Z(ru)=N(u),\;r\ge1
\]
and estimating the inhomogeneous term in $L^1_t L^2_x$, 
implies for any time interval $I\ni0$
\EQ{
\|u\|_{L^3_t(I;L^6_x)} + \|\vec u\|_{L^\infty_t;\HH} \lec \|\vec u(0)\|_{\HH} + \|u\|^3_{L^3_t(I;L^6_x)}
}
By the usual continuity argument (expanding $I$) this implies  $$\|\vec u(0)\|_{\HH}<\delta \Longrightarrow \|u\|_S\lec\delta$$
Moreover, the scattering is also standard. Indeed, denoting the free   propagator in $\R^5_*$ with a Dirichlet boundary condition again by $S_0(t)$, we
seek $\vec v(0)\in\HH$ such that $$\vec u(t)=S_0(t)\vec v(0)+o_{\HH}(1)$$ as $t\to\infty$. In view of the Duhamel representation of
$\vec u$ and using the group property and unitarity of $S_0$ this is tantamount to
\EQ{
\vec v(0) = \vec u(0)+ \int_0^\infty S_0(-s) (0,N(u(s)))\,  ds
}
The integral on the right-hand side is absolutely convergent in $\HH$ provided $\|u\|_S<\infty$. The necessity of the latter condition
follows from the fact that free waves satisfy it, whence by the small data theory (applied to large times) it carries over to any nonlinear
wave that scatters. 
\end{proof}

We remark that in the $\psi$ formulation, the scattering of Proposition~\ref{prop:2} means precisely~\eqref{psi scat}, \eqref{fy eq}. 

To prove Theorem~\ref{main} we therefore need to show that every energy solution~$\psi$ of~\eqref{WM} has the property that in the $u$-formulation
$\|u\|_S<\infty$. This will be done by means of the Kenig-Merle concentration-compactness approach~\cite{KM1}, \cite{KM2}. 

\section{Concentration Compactness}

In this section, we prove the following result. 

\begin{prop}
\label{prop:3}
Suppose that Theorem~\ref{main} fails. 
Then there exists a  nonzero energy solution to~\eqref{WM} (referred to as critical element) $\vec\psi(t)$ for $t\ge0$ with the property
that the trajectory $$\calK_+:= \{\vec \psi(t) \mid t\ge0\} $$ is precompact in $\HH$. 
\end{prop}
 
In the following section we then lead this to a contradiction via a virial-type rigidity argument. 
To prove Proposition~\ref{prop:3} we may work in the $u$-formulation of equation~\eqref{ueq}
since the map $u=r^{-1}\psi$ is an isomorphism between  $\HH$ in   $\R^5_*$ and $\R^3_*$, respectively. 

To proceed, we need the following version of the Bahouri-G\'erard decomposition~\cite{BaG}.  As before, ``free'' waves refer to solutions of~\eqref{v free}. 
The following two lemmas are standard, see in particular Chapter~2 of the book~\cite{NakS}.  

\begin{lem}\label{lem:BG}
Let $\{u_n\}$ be a sequence of free radial waves bounded in $\HH=\dot H^1_{0}\times L^2(\R^5_*)$. Then after replacing it by a subsequence, 
there exist a sequence of free solutions $v^j$ bounded in $\HH$, and sequences of times $t_n^j\in\R$ such that for  $\ga_n^k$ defined by
\EQ{\label{eq:BGdecomp}
  u_n(t) = \sum_{1\le j<k} v^j(t+t_n^j) + \ga_n^k(t) }
we have for any $j<k$, $\vec\ga_n^k(-t_n^j) \rightharpoonup 0$ weakly in $\HH$ as $n\to\I$,  as well as 
\EQ{\label{eq:tdiverge}
 \lim_{n\to\I} |t_n^j-t_n^k| = \I     
}
and the errors $\ga_{n}^{k}$ vanish asymptotically
in the sense that
\EQ{ \label{gammavanish}
 \pt \lim_{k\to \I} \limsup_{n\to\I} \|\ga_n^k\|_{(L^\I_tL^p_x\cap L^3_t L^6_x)(\R\times\R^5_*)}=0 \quad \forall \; \frac{10}{3}<p<\infty 
}
Finally, one has orthogonality of the free energy
\EQ{ \label{H1 orth}
 \| \vec u_n \|_{\HH}^2 = \sum_{1\le j<k} \|\vec v^j\|_{\HH}^2 + \|\vec\ga_n^k\|_{\HH}^2 +o(1)
} 
as $n\to\I$. 
\end{lem}

\begin{proof}
Recall the Sobolev embeddings $\dot H^1_{0}(\R^5_*)\embed L^{\frac{10}{3}}\cap L^\infty (\R^5_*)$ for radial functions. 
Moreover, for any $p\in(\frac{10}{3},\infty)$ the embedding is compact. 
Since $\ga_n^k$ is bounded in $\dot H^1_{0}$, interpolation with these, as well as the Strichartz  estimates from~\cite{SmSo}  implies that it suffices to bound the 
remainder in $L^\I_tL^p_x$ for any fixed $ p\in(\frac{10}{3},\infty)$. Fix such a~$p$.  
Let $\ga_n^0:=u_n$ and $k=0$. If 
\EQ{\nn 
 \nu^k:=\limsup_{n\to\I}\|\ga_n^k\|_{L^\I_t L^p_x}=0,} 
then we are done by putting $\ga_n^\ell=\ga_n^k$ for all $\ell>k$. 
Otherwise, there exists a sequence $t_n^k\in\R$ such that $\|\ga_n^k(-t_n^k)\|_{L^p_x} \ge \nu^k/2$ for large $n$. 
Since $\vec\ga_n^k(-t_n^k)\in \HH$ is bounded, after extracting a subsequence it converges weakly in $\HH$, and $\ga_n^k(-t_n^k)$ 
converges strongly in $L^p_x(\R^5_*)$. Let $v^k$ be the free wave given by the limit
\EQ{\nn 
 \lim_{n\to\I}\vec\ga_n^k(-t_n^k) = \vec v^k(0)}
By Sobolev $\|v^k(0)\|_{\dot H^1_{0}(\R^5_*)}\gec\nu^k$. We repeat the same procedure inductively in $k\ge 1$. 
As before, let $S_{0}(t)$ denote the free  exterior propagator in $\HH$. If $t_n^j-t_n^k\to c\in\R$ for some $j<k$, then 
\EQ{\nn 
 \vec\ga_n^k(-t_n^k)=S_{0}(t_n^j-t_n^k)\vec\ga_n^k(-t_n^j) \to 0,}
weakly in $\HH$. To see this, it suffices to show that $$\lan \vec\ga_n^k(-t_n^k)\mid\vec \phi\ran\to0\qquad n\to\I$$ for any Schwartz function $\vec\phi$.
But  one has 
\[
\lan \vec\ga_n^k(-t_n^k)\mid\vec \phi\ran = \lan \vec\ga_n^k(-t^{j}_{n})\mid S_{0}(t_n^k-t_n^j)\vec \phi\ran \to0 
\]
since $S_{0}(t_n^k-t_n^j)\vec \phi \to S_{0}(-c) \vec\phi$ strongly in $L^{2}$. 
Hence $|t_n^j-t_n^k|\to\I$ as long as $\vec v^k\not=0$. Then for all $j\le k$, 
\EQ{\nn 
 \vec\ga_n^{k+1}(-t_n^j)=\vec\ga_n^k(-t_n^j)-\vec v^k(t_n^k-t_n^j) \rightharpoonup 0}
weakly in $\HH$. Indeed, if $j<k$ then this follows from the inductive assumption, whereas for $j=k$ it follows by construction. 

To prove  \eqref{H1 orth}, expand (without loss of generality at $t=0$) 
\[
\| \vec u_n(0) \|_{\HH}^2 = \Big\| \sum_{1\le j<k} \vec v^j(t_n^j) + \vec \ga_n^k(0)\Big \|_{\HH}^{2}
\]
The cross terms are all $o(1)$ as $n\to\I$:  for $k>j\ne\ell$, and with the scalar product in $\HH$, 
\EQ{\label{cross terms}
\lan \vec v^j(t_n^j) \mid \vec v^\ell(t_n^\ell) \ran &= \lan \vec v^j(0) \mid S_{0}(t_n^\ell-t_{n}^{j})\vec v^\ell(0) \ran \to0 \\
\lan \vec v^j(t_n^j) \mid \vec \ga_n^k (0)\ran &= \lan \vec v^j(0) \mid \vec \ga_n^k (-t_{n}^{j})\ran \to0
}
The first line of~\eqref{cross terms}
vanishes as $n\to\I$ due to $\| S_{0}(t_n^\ell-t_{n}^{j}) \vec\phi\|_{\I} \to 0$ for any Schwartz function $\vec\phi$ since $|t_n^\ell-t_{n}^{j}|\to\I$, by 
the pointwise decay of free waves with Schwartz data; as usual this suffices since we can approximate $\vec v^{j}(0), \vec v^{\ell}(0)$ by Schwartz functions. 
The second line vanishes by $\vec\ga_n^k (-t_{n}^{j})\rightharpoonup0$ in $\HH$ as $n\to\I$. 

Finally, one uses \eqref{H1 orth} to conclude that $\nu^{j} \to0$: 
\EQ{\nn 
 \limsup_{n\to\I}\| \vec u_n\|_{\HH}^2 \ge \sum_{j<k}\| \vec v^j\|_{\HH}^2
 \gec \sum_{j<k}(\nu^j)^2}
uniformly in $k$. The final inequality follows from the radial Sobolev embedding (in other words,
Sobolev embedding and compactness). 
Hence, $\limsup_{n\to\I}\|\ga_n^k\|_{L^\I_t L^p_x}=\nu^k \to 0$, as $k\to\I$. 
\end{proof}

Applying this decomposition to the nonlinear equation requires a perturbation lemma which we now formulate. All spatial norms are 
understood to be on~$\R^5_*$. The exterior propagator $S_0(t)$ is as above. 

\begin{lem} \label{PerturLem} 
There are continuous functions $\eps_0,C_0:(0,\I)\to(0,\I)$ such that the following holds: 
Let $I\subset \R$ be an open interval (possibly unbounded), $u,v\in C(I;\dot H_{0}^{1})\cap C^{1}(I;L^{2})$ radial functions satisfying for some $A>0$ 
\EQ{\nn 
\|\vec u\|_{L^\infty(I;\HH)} +  \|\vec v\|_{L^\infty(I;\HH)} +   \|v\|_{L^3_t(I;L^6_x)} & \le A \\   
 \|\glei(u)\|_{L^1_t(I;L^2_x)} 
   + \|\glei(v)\|_{L^1_t(I;L^2_x)} + \|w_0\|_{L^3_t(I;L^6_x)} &\le \eps \le \eps_0(A),}
where $\glei(u):=\Box u+u^3 Z(ru)$ in the sense of distributions, and $\vec w_0(t):=S_{0}(t-t_0)(\vec u-\vec v)(t_0)$ with $t_0\in I$ arbitrary but fixed.  Then
\EQ{ \nn 
  \|\vec u-\vec v-\vec w_0\|_{L^\I_t(I;\HH)}+\|u-v\|_{L^3_t(I;L^6_x)} \le C_0(A)\eps.} 
  In particular,  $\|u\|_{L^3_t(I;L^6_x)}<\I$. 
\end{lem}
\begin{proof}
Let $X:=L^3_tL^6_x$ and 
\EQ{\nn 
 w:=u-v, \pq e:=\Box (u-v)+u^3 Z(ru)-v^3Z(rv) = \glei(u) - \glei(v).} 
There is a partition of the right half of $I$ as follows, where $\delta_{0}>0$ is a small
absolute constant which will be determined below: 
\EQ{\nn 
 \pt t_0<t_1<\cdots<t_n\le \I,\pq I_j=(t_j,t_{j+1}),\pq I\cap(t_0,\I)=(t_0,t_n),
 \pr \|v\|_{X(I_j)} \le \de_{0} \pq(j=0,\dots,n-1), \pq n\le C(A,\de_{0}).}
We omit the estimate on $I\cap(-\I,t_0)$ since it is the same by symmetry. 
Let $\vec w_j(t):=S_{0}(t-t_j)\vec w(t_j)$ for all $0\le j <n$.  Then
\EQ{\label{w form}
\vec w(t)  
&= \vec w_{0}(t) + \int_{t_0}^{t} S_0(t-s) (0,e-(v+w)^{3} Z(r(v+w)) + v^{3}Z(rv))(s)\, ds
}
which implies that, for some absolute constant $C_{1}\ge1$, 
\EQ{ \label{eq:ww0}
\pn \| w-w_{0}\|_{X(I_0)} 
 \pt\lec  \|(v+ w)^3 Z(r(v+w))- v^3 Z(rv)-e\|_{L^1_tL^2_x(I_0)}
 \pr\le C_{1} (\de_{0}^{2}+\| w\|_{X(I_0)}^{2})\| w\|_{X(I_0)}+C_{1}\eps
}
To estimate the differences involving the $Z$ function we invoke its smoothness as well as the fact that by radiality, $ru$ and $rv$
are bounded pointwise in terms of the energy of $u$ and $v$, respectively  (which we assume to be bounded by~$A$). 
Note that $\|w\|_{X(I_{0})}<\I$ provided $I_{0}$ is a finite interval. If $I_{0}$ is half-infinite, then we first
need to replace it with an interval of the form $[t_{0},N)$, and  let $N\to\I$ after performing the estimates which are
uniform in~$N$.  Now assume that $C_{1}\delta_{0}^{2}\le \frac14$ and fix $\delta_{0}$ in this fashion.
By means of the continuity method (which refers to using that the $X$-norm is continuous in  the upper endpoint of $I_{0}$), 
\eqref{eq:ww0} implies that $\| w\|_{X(I_{0})}\le 8C_{1}\eps$. 
Furthermore, Duhamel's formula implies  that 
\EQ{\nn
\vec w_{1}(t)- \vec w_{0}(t) = \int_{t_{0}}^{t_{1}} S_0(t-s) (0,e-(v+w)^{3} Z(r(v+w)) + v^{3}Z(rv))(s)\, ds
}
whence also
\EQ{\label{eq:w1w0}
\| w_{1}-w_{0}\|_{X(\R)} \lec \int_{t_0}^{t_1} \| (e-(v+w)^{3} Z(r(v+w)) + v^{3}Z(rv))(s)\|_{2}\, ds
}
which is estimated as in~\eqref{eq:ww0}.  We conclude that $ \| w_{1}\|_{X(\R)}\le 8C_{1}\eps$. 
In a similar fashion one verifies that for all $0\le j<n$ 
\EQ{ \label{est S'}
 \pn\| w- w_j\|_{X(I_j)} + \| w_{j+1}-w_j\|_{X(\R)}
 \pt\lec  \|  e-(v+w)^{3} Z(r(v+w)) + v^{3}Z(rv)  \|_{L^1_tL^2_x(I_j)}
 \pr\le C_{1} (\de_{0}^{2}+\| w\|_{X(I_j)}^{2})\| w\|_{X(I_j)}+C_{1}\eps}
 where $C_{1}\ge1$ is as above. 
By induction in $j$ one obtains that
 \EQ{\nn
 \| w\|_{X(I_{j})} + \| w_{j}\|_{X(\R)} \le C(j)\, \eps\quad \forall \; 1\le j<n
 }
 This requires that  $\eps<\eps_{0}(n)$ which can be done provided $\eps_0(A)$ is chosen small enough.
Repeating the estimate~\eqref{est S'} once more,  but with the energy piece $L^{\I}_{t}\HH$ included on the left-hand side, 
we can now bound the $S(I)$-norm on~$w$. 
\end{proof}

We can now apply standard arguments to prove the main result of this section. Without further mention, all functions are radial. 

\begin{proof}[Proof of Proposition~\ref{prop:3}]
Suppose that the theorem fails. Then 
there exists a bounded sequence $\vec u_n:=(u_{0,n},u_{1,n})\in \HH$ with 
\[
\| \vec u_n\|_{\HH} \to E_*>0,\quad \|u_n\|_S\to\infty
\]
where $u_n$ denotes the global evolution of $\vec u_n$ of~\eqref{ueq}. We may assume that $E_*$ is minimal with this property. 
Applying Lemma~\ref{lem:BG} to the free evolutions of $\vec u_n(0)$ 
yields  free waves $v^{j}$ and times $t^{j}_{n}$ as in~\eqref{eq:BGdecomp}.  
Let $U^{j}$ be the nonlinear profiles of $(v^{j}, t^{j}_{n})$, i.e., those energy solutions of~\eqref{ueq} which satisfy
\EQ{\nn
\lim_{t\to t^{j}_{\I} }\| \vec v^{j}(t) - \vec U^{j}(t)\|_{\HH} \to0
}
where $\lim_{n\to\I} t^{j}_{n} = t^{j}_{\I}\in [-\I,\I]$.  The $U^{j}$ exist locally around $t=t^{j}_{\I}$ 
by the local existence and scattering theory, see Proposition~\ref{prop:2}.  
Locally around $t=0$ one has the following {\em nonlinear profile decomposition}
\EQ{ \label{eq:nonlinearprofile} 
u_{n} (t) =  \sum_{j<k}  U^{j}(t+t^{j}_{n}) + \ga_{n}^{k}(t) + \eta_{n}^{k}(t)
}
where $\| \vec \eta_{n}^{k}(0)\|_{\HH}\to0$ as $n\to\I$.  
Now suppose that either there are two non-vanishing $v^{j}$, say
$v^{1}, v^{2}$, or that 
\EQ{\label{eq:ganonvanish}
\limsup_{k\to\I}\limsup_{n\to\I} \| \vec \ga^{k}_{n}\|_{\HH}>0
}
Note that the left-hand side does not depend on time since $\ga_{n}^{k}$ is a free wave. 
By the minimality of $E_{*}$ and the orthogonality of the energy~\eqref{H1 orth}
each $U^{j}$ is a global solution and scatters with $\| U^{j}\|_{L^{3}_t L^{6}_{x}} <\I$.   

We now apply Lemma~\ref{PerturLem} on $I=\R$ with $u=u_{n}$ and 
\EQ{ \label{eq:sumUj}
v(t)=\sum_{j<k} U^{j}(t+t^{j}_{n})
}
That $\| \glei(v)\|_{L^{1}_{t} L^{2}_{x}}$ is small for large $n$ follows from~\eqref{eq:tdiverge}. 
To see this, note that with $N(v):=v^3 Z(rv)$, 
\EQ{\nn
\glei(v) &= \Box v+v^{3} Z(rv) \\
& = -\sum_{j<k} N(U^{j}(t+t^{j}_{n}))    + N\big( \sum_{j<k} U^{j}(t+t^{j}_{n}) \big)
}
The difference on the right-hand side here only consists of terms which involve at least one pair of distinct $j,j'$. 
But then $\| \glei(v)\|_{L^{1}_{t} L^{2}_{x}}\to 0$ as $n\to\I$ by~\eqref{eq:tdiverge}.  
In order to apply Lemma~\ref{PerturLem}   it is essential  that 
\EQ{\label{eq:kunif}
\limsup_{n\to\I} \big\| \sum_{j<k} U^{j}(t+t^{j}_{n})  \big\|_{L^{3}_t L^{6}_{x}} \le A < \I
}
{\em uniformly in} $k$, which follows from \eqref{eq:tdiverge}, \eqref{H1 orth}, and Proposition~\ref{prop:2}. 
The point here is that the sum can be split into one over $1\le j<j_{0}$ and another over $j_{0}\le j<k$.
This splitting is performed in terms of the energy, with $j_{0}$ being chosen such that for all $k>j_{0}$
\EQ{\label{eq:Ujeps0}
\limsup_{n\to\I}\sum_{ j_{0} \le j < k}\| \vec U^{j}(t^{j}_{n})\|_{\HH}^{2} \le \eps_{0}^{2}
}
where $\e_{0}$ is fixed such that the small data result of Proposition~\ref{prop:2} applies. 
Clearly, \eqref{eq:Ujeps0}  follows from~\eqref{H1 orth}. Using~\eqref{eq:tdiverge} as well
as the small data scattering theory   one now obtains 
\EQ{\label{eq:L3 orth}
\limsup_{n\to\I}\Big\|  \sum_{ j_{0} \le j < k}  U^{j}(\cdot+t^{j}_{n})  \Big\|_{L^{3}_{t}L^{6}_{x}}^{3} &=\sum_{ j_{0} \le j < k}  \big\|   
U^{j}(\cdot)  \big\|_{L^{3}_{t}L^{6}_{x}}^{3} \\ 
& \le C \limsup_{n\to\I} \Big( \sum_{ j_{0} \le j < k} \| \vec U^{j}(t^{j}_{n})\|_{\HH}^{2}  \Big)^{\f32}
}
with an absolute constant $C$. This implies \eqref{eq:kunif}, uniformly in~$k$. 

Hence one can take $k$ and $n$  so large that Lemma~\ref{PerturLem}  applies to~\eqref{eq:nonlinearprofile} whence
\EQ{\nn
\limsup_{n\to\I} \|u_{n}\|_{L^{3}_t L^{6}_{x}} < \I
}
which is a contradiction. Thus, there can be only one nonvanishing $v^{j}$, say $v^{1}$, and  moreover 
\EQ{\label{eq:gavanish}
 \limsup_{n\to\I} \| \vec \ga^{2}_{n}\|_{\HH} = 0
}
Thus, $\|\vec U^{1}\|_{\HH}=E_{*}$.  By the preceding, necessarily 
\EQ{\label{eq:U1Stinf}
\|U^1\|_{L^{3}_t L^{6}_{x}} = \I
}
Therefore, $U^1=:u_*$ is the desired critical element. Suppose that 
\EQ{\label{eq:U1Stinf+}
\|u_*\|_{L^{3}_t([0,\I); L^{6}_{x})} = \I
}
Then we claim that $$\calK_+:= \{\vec u_*(t) \mid t\ge0\} $$ is precompact in~$\HH$. If not, then there exists $\delta>0$ so that 
for some infinite sequence $t_n\to\I$ one has
\EQ{\label{eq:U1delta}
\| \vec u_*(t_n)- \vec u_*(t_m)\|_\HH>\delta \quad \forall \; n>m
}
Applying Lemma~\ref{lem:BG} to $U^1(t_n)$ one concludes via the same argument as before
based on the minimality of $E_*$ and \eqref{eq:U1Stinf} that 
\EQ{\label{eq:U1BG}
\vec u_*(t_n) = \vec V (\tau_n) + \vec \gamma_n(0)
}
where $\vec V$, $\vec \gamma_n$ are free   waves in~$\HH$, and $\tau_n$ is some
sequence in~$\R$. Moreover, $\| \vec\gamma_n\|_\HH\to0$ as $n\to\I$. 
If $\tau_n\to \tau_\I\in\R$, then \eqref{eq:U1BG} and~\eqref{eq:U1delta} lead to a contradiction.
If $\tau_n\to\I$, then 
\EQ{\nn
\| V(\cdot+\tau_n)\|_{L^{3}_t([0,\I); L^{6}_{x})}  \to0 \qquad \text{\ as\ }n\to\I
}
implies via the local wellposedness theory
that $\| u_*(\cdot+t_n)\|_{L^{3}_t([0,\I); L^{6}_{x})}<\I$ for all large $n$, which is a contradiction to~\eqref{eq:U1Stinf+}.
If $\tau_n\to-\I$, then  
\EQ{\nn
\| V(\cdot+\tau_n)\|_{L^{3}_t( (-\I,0]; L^{6}_{x})}  \to0 \qquad \text{\ as\ } n\to\I
}
implies that $\| u_*(\cdot+t_n)\|_{L^{3}_t( (-\I,0] ; L^{6}_{x})}<C<\I$ for all large $n$ where $C$ is some fixed constant. 
Passing to the limit yields a contradiction to~\eqref{eq:U1Stinf} and~\eqref{eq:U1delta} is seen to be false, concluding the proof of compactness
of~$\calK_+$. 
\end{proof}

\section{The rigidity argument}\label{The rigidity argument}

In this section we complete the proof of Theorem~\ref{main} by showing that a critical element as given by Proposition~\ref{prop:3}
does not exist.  This is based on the virial identity exterior to the ball. The main novelty here lies with the fact that due to the radial assumption in $\R^3_*$ 
we are able to show that the nonlinear functional arising in this virial identity is {\em globally coercive} on the energy space. In contrast, for equivariant energy 
critical wave maps in the energy class, C\^{o}te, Kenig, Merle~\cite{CKM} needed an upper bound on the energy in order to apply the virial argument. In particular, we have the following proposition.

\begin{prop}[Rigidity Property]  \label{rigid}
Let $(\psi_0, \psi_1) \in \HH$, and denote by $\vec \psi(t)$ the associated global in time solution to \eqref{WM} given by Proposition \ref{prop:1}. Suppose that the trajectory
\begin{align*}
\calK_{+}:=\{ \vec \psi(t)\mid t \ge 0\}
\end{align*}
is precompact in  $\HH$. Then $\psi \equiv 0$. 
\end{prop}
 
 The proof of Proposition~\ref{rigid} relies on the following two results related to the virial identity for solutions to \eqref{WM}.   
 In what follows we let $\chi \in C^{\I}_0(\R)$ be an even function so that $\chi(r) =1$ for $\abs{r} \le 1$, $\supp (\chi) \in [-2 ,2]$ and  $\chi(r) \in [0,1]$ for every $r \in \R$.  Define $\ds{\chi_R(r):= \chi(R^{-1}r)}$.
 
 \begin{lem} Let $\vec \psi(t) \in \HH$ be a solution to \eqref{WM}. Then, for every $T \in \R$ we have 
  \begin{align} \label{virial}
  \ang{\chi_R \dot{\psi}\vert r\psi_r} \Big|_0^{T}   &\le  \int_0^T \left\{ -\frac{3}{2}\int_1^{\I}  \dot{\psi}^2\, r^2 \, dr + 
  \frac{1}{2} \int_1^{\I} \psi_r^2\, r^2 \, dr\right\} \, dt  \\ \notag
  &\quad + \int_0^T\left\{\int_1^{\I} \sin^2(\psi) \, dr + O( \E_R^{\I}(\vec \psi)) \right\} \, dt\\
\label{mass}
 \ang{\chi_R \dot{\psi}\vert \psi} \Big|_0^{T}  &=  \int_0^T \left\{ \int_1^{\I}  \dot{\psi}^2\, r^2\, dr - 
   \int_1^{\I} \psi_r^2\, r^2 \, dr -\int_1^{\I} \psi \sin(2\psi) \, dr \right\} \, dt  \\ \notag
  &\quad + \int_0^T\left\{ O( \E_R^{\I}(\vec \psi)) + O\left(\int_R^{\infty} \psi^2\, dr\right) \right\} \, dt
\end{align}
where here, the brackets $\ang{\cdot \vert \cdot}$ refer to the $L^2_{\textrm{rad}}(\R^3_{*})$ pairing $\ds{\ang{f\vert g}:=  \int_1^{\infty} f(r)g(r) r^2 \,dr}$ and
\begin{align}
 \E_R^{\I}(\vec \psi):= \frac{1}{2}\int_R^{\I} \left(\dot{\psi}^2 + \psi_r^2 + \frac{2\sin^2(\psi)}{r^2}\right) r^2 \, dr
 \end{align} 
 \end{lem}

\begin{proof} We first establish  \eqref{virial} for solutions $\vec \psi(t) \in C^{\infty}_0\times C^{\infty}_0(\R^3_*)$. 
 \begin{align*}
 \frac{d}{dt}  \ang{\chi_R \dot{\psi} \,\vert\,  r\psi_r} &=  \ang{\chi_R\ddot{\psi}\,\vert\, r \psi_r} + \ang{\chi_R \dot \psi \,\vert \,r \dot \psi_r}\\
 \\
 &= \ang{\chi_R\left(\psi_{rr} + \frac{2}{r} \psi_r - \frac{\sin(2\psi)}{r^2}\right) \,\vert \,r \psi_r} + \ang{\chi_R \dot \psi \,\vert \,r \dot \psi_r}\\
 \\
 &= \frac{1}{2} \int_1^{\infty}  \p_r( \psi_r^2) ( \chi_R r^3) \, dr + 2 \int_1^{\I} \chi_R\, \psi^2_r  r^2 \, dr \\
 & \quad - \int_1^{\I} \p_r( \sin^2(\psi)) \chi_R r\, dr + \frac{1}{2} \int_1^{\I}  \p_r(\dot \psi^2) \chi_R r^3 \, dr\\
 \end{align*}
 Integrating by parts, the preceding line can be further simplified as follows:
 \begin{align*}
 &= -\frac{3}{2} \int_1^{\infty} \chi_R \dot \psi^2 r^2 \, dr +  \frac{1}{2} \int_1^{\infty} \chi_R \psi_r^2   r^2 \, dr + \int^{\infty}_1 \chi_R \sin^2(\psi) \, dr - \frac{1}{2} \psi_r^2(t, 1)  \\
 & \quad + \frac{1}{2}\int_1^{\I}\left( \psi_r^2 - \dot \psi^2 + \frac{2\sin^2(\psi)}{r^2} \right) r \, \chi^{\prime}_R\,  r^2 \, dr
 \\
 \\
 &= -\frac{3}{2} \int_1^{\infty}  \dot \psi^2 r^2 \, dr +  \frac{1}{2} \int_1^{\infty}  \psi_r^2 \,  r^2 \, dr + \int^{\infty}_1  \sin^2(\psi) \, dr - \frac{1}{2} \psi_r^2(t, 1)  \\
&- \int_1^{\infty} (1- \chi_R) \left( -\frac{3}{2}  \dot \psi^2 r^2  +  \frac{1}{2}  \psi_r^2 \,  r^2  +  \frac{ \sin^2(\psi)}{r^2} \right)r^2 \, dr
 \\
 & \quad + \frac{1}{2} \int_1^{\I}\left( \psi_r^2 - \dot \psi^2 + \frac{2\sin^2(\psi)}{r^2} \right) r \, \chi^{\prime}_R\,  r^2 \, dr
\end{align*} 
Next, observe that 
\begin{align*}
\abs{ \int_1^{\infty} (1- \chi_R) \left( -\frac{3}{2}  \dot \psi^2 r^2  +  \frac{1}{2}  \psi_r^2 \,  r^2  +  \frac{ \sin^2(\psi)}{r^2} \right)r^2 \, dr} \lec \E_R^{\infty}(\vec \psi)
\end{align*}
And similarly, since $\supp(\chi'(R^{-1} \cdot)) \cap [1, \infty)  \subset [R, 2R]$, we have
\begin{align*}
&\abs{ \frac{1}{2} \int_1^{\I}\left( \psi_r^2 - \dot \psi^2 + \frac{2\sin^2(\psi)}{r^2} \right) r \, \chi^{\prime}_R\,  r^2 \, dr}\\
& \le  \frac{1}{2} \int_1^{\I}\left( \psi_r^2 + \dot \psi^2 + \frac{2\sin^2(\psi)}{r^2} \right)R^{-1} r \, \abs{\chi^{\prime}(R^{-1}r)}\,  r^2 \, dr \\
& \lec  \E_R^{\infty}(\vec \psi)
\end{align*}
Putting this together, we obtain 
\begin{align*}
 \frac{d}{dt}  \ang{\chi_R \dot{\psi} \,\vert \, r\psi_r} & =-\frac{3}{2} \int_1^{\infty}  \dot \psi^2 r^2 \, dr +  \frac{1}{2} \int_1^{\infty}  \psi_r^2 \,  r^2 \, dr + \int^{\infty}_1  \sin^2(\psi) \, dr \\
 &\quad - \frac{1}{2} \psi_r^2(t, 1) + O(\E_R^{\infty}(\vec \psi)) 
 \\
 \\
 &\le -\frac{3}{2} \int_1^{\infty}  \dot \psi^2 r^2 \, dr +  \frac{1}{2} \int_1^{\infty}  \psi_r^2 \,  r^2 \, dr \\
 &\quad+ \int^{\infty}_1  \sin^2(\psi) \, dr +O(\E_R^{\infty}(\vec \psi))
\end{align*}
By integrating the above inequality in time from $0$ to $T$ we obtain \eqref{virial} for smooth solutions.  Our well-posedness theory  for \eqref{WM} then  allows us to extend \eqref{virial} to all energy class solutions $\vec \psi (t) \in \HH$ via an approximation argument. 

We proceed in a similar fashion to prove \eqref{mass}.  Thus, for smooth $\psi$ we have by direct calculation, 
\begin{align*}
 \frac{d}{dt} \ang{\chi_R \dot{\psi}\mid \psi} &= \ang{\chi_R \ddot{\psi}\,\vert\, \psi} + \ang{\chi_R \dot{\psi}\,\vert\, \dot\psi}\\
 \\
 &=  \ang{\chi_R\left(\psi_{rr} + \frac{2}{r} \psi_r - \frac{\sin(2\psi)}{r^2}\right)\,\vert \,\psi} + \ang{\chi_R \dot{\psi}\,\vert\, \dot\psi}\\
 \\
 &=\ang{\frac{\chi_R}{r^{2}}\p_r\left(r^2 \psi_r\right)\, \vert\,  \psi}- \ang{  \chi_R \frac{\sin(2\psi)}{r^2}\,\vert \, \psi} + \ang{\chi_R \dot{\psi}\, \vert\, \dot\psi}\\
 \end{align*}
 Integrating by parts, the above simplifies as follows:
 \begin{align*}
 &=  \int_1^{\infty}  \chi_R \dot \psi^2 r^2 \, dr -   \int_1^{\infty} \chi_R \psi_r^2 \,  r^2 \, dr - \int^{\infty}_1\chi_R  \psi \sin(2\psi) \, dr \\
 &\quad - \int_1^{\infty} \psi_r \psi \chi_R^{\prime} r^2\, dr
 \\
 &=\int_1^{\infty}   \dot \psi^2 r^2 \, dr -   \int_1^{\infty}  \psi_r^2 \,  r^2 \, dr - \int^{\infty}_1  \psi \sin(2\psi) \, dr \\
 &\quad -\int_1^{\infty}(1-\chi_R) \left( \dot \psi^2  - \psi_r^2  \right)\,r^2\, dr\\  
 &\quad +\int_1^{\infty}\left\{ (1-\chi_R)  \psi \sin(2\psi)+\frac{1}{2} \psi^2 \p_r( \chi_R^{\prime} r^2)\right\}\, dr
\end{align*}
As before we have, 
\begin{align*}
\abs{-\int_1^{\infty}(1-\chi_R) \left( \dot \psi^2  - \psi_r^2\right)\,r^2\, dr } \le \E_R^{\I} (\vec \psi)
\end{align*}
And, since $\abs{\psi \sin(2 \psi)} \le 2 \psi^2 $, we can deduce that 
\begin{align*}
&\abs{\int_1^{\I}\left\{(1-\chi_R)  \psi \sin(2\psi)   +\frac{1}{2}  \psi^2 \p_r( \chi_R^{\prime} r^2)\right\}\, dr} \\
&\lec \int_1^{\infty} (1-\chi_R)\psi^2\, dr
 + \int_1^{\infty} \psi^2 \abs{\chi'(R^{-1}r)} R^{-1}r\, dr
 + \int_1^{\I}\psi^2 \abs{\chi''(R^{-1}r)} R^{-2}r^2\, dr\\
&\lec  \int_R^{\infty} \psi^2 \, dr
\end{align*}
Therefore, we see thatf
\begin{align*}
\frac{d}{dt} \ang{\chi_R \dot{\psi}\vert \psi} &= \int_1^{\infty}   \dot \psi^2 r^2 \, dr -   \int_1^{\infty}  \psi_r^2 \,  r^2 \, dr - \int^{\infty}_1  \psi \sin(2\psi) \, dr \\
&\quad+ O\left( \E_R^{\I}(\vec \psi)\right) + O\left(\int_R^{\infty} \psi^2\, dr\right)
\end{align*}
Integrating the above in time from $0$ to $T$ proves \eqref{mass} for smooth solutions. Approximating energy solutions by smooth solutions concludes the proof.  
\end{proof}

From \eqref{virial} and \eqref{mass} we construct a nonlinear functional, $\LL: \HH \to \R$, whose global coercivity on  $\HH$ is a key ingredient in the proof of Theorem \ref{rigid}.  Using Lemma \ref{virial} we consider the following linear combination of \eqref{virial} and \eqref{mass}:
\begin{align}\label{virial ineq}
 \ang{\chi_R \dot{\psi}\, \vert \, r\psi_r +  \frac{29}{20}  \psi}   \Big|_0^T  &\le -\int_0^T \left[ \int_1^{\I} \left(\frac{1}{20}   \dot \psi^2   + \frac{19}{20}  \psi_r^2\right) \,  r^2 \, dr\right]   dt\\ \notag
 &\quad + \int_0^T \left[ \int_1^{\infty} \left(\sin^2 (\psi) - \frac{29}{20} \psi \sin(2 \psi) \right)\, dr \right] \, dt\\ \notag
 &\quad + \int_0^T \left[ O\left( \E_R^{\I}(\vec \psi)\right) + O\left(\int_R^{\infty} \psi^2\, dr\right) \right] \, dt
\end{align}
 We define $\LL: \HH \to \R$ as follows
  \begin{align} \label{funct}
  \LL(\vec \psi):=  -\int_1^{\I} \left(\frac{1}{20}   \dot \psi^2   + \frac{19}{20}  \psi_r^2\right) \,  r^2 \, dr +  \int_1^{\infty} \left(\sin^2 (\psi) - \frac{29}{20} \psi \sin(2 \psi) \right)\, dr
  \end{align}
  
  \begin{lem}  \label{coercive}Let $\LL: \HH \to \R$ be defined as in \eqref{funct}. Then for every $\vec \psi= (\psi(t), \dot \psi(t)) \in \HH$ we have 
  \begin{align}\label{coerce}
  \LL( \vec \psi) \le -\frac{1}{20} \int_1^{\infty} \left( \dot \psi^2 + \psi_r^2 \right) r^2 \, dr \le -\frac{1}{180} \E(\vec \psi)
  \end{align}
  \end{lem}
 We postpone the proof of Lemma~\ref{coercive}, and first use it to prove Proposition~\ref{rigid}.

\begin{proof}[Proof of Proposition~\ref{rigid}] Suppose that $\vec \psi(t) \in \HH$ satisfies the conditions of Proposition~\ref{rigid}, i.e., suppose that 
\begin{align*}
K_{+}: = \{ \vec \psi(t) \mid t \ge 0\} 
\end{align*} 
is pre-compact in $\HH$. Note that the pre-compactness of $K_{+}$ in $\HH$ implies, by Hardy's inequality, that $K_{+}$ is also pre-compact in $L^2(\R^3_*,  dr)$ where 
\begin{align*}
\|\psi(t) \|_{L^2(\R^3_*, dr)}^2 := \int_1^{\infty} \psi(t)^2 \, dr
\end{align*}
Then, for every  $\e>0$ there exists $R(\e)$ such that for every $t\ge0$ we have 
\begin{align}\label{small tail}
\E_{R(\e)}^{\I}(\vec \psi(t)) + \int_{R(\e)}^{\infty} \psi(t)^2 \, dr < \e
\end{align}
Now, by \eqref{virial ineq} and Lemma~\ref{coercive}, we have that for all $T$ 
\begin{align*}
  \ang{\chi_R \dot{\psi}\, \vert \, r\psi_r +  \frac{29}{20}  \psi} \Big|_0^T  &\le \int_0^T\left[ \LL(\vec \psi)+ O\left(\E_R^{\I}(\vec \psi(t))+ \int_R^{\infty} \psi(t)^2\, dr\right)  \right]\, dt \\
 \\
 & \le \int_0^T \left[- \frac{\E(\vec \psi)}{180}  +  O\left( \E_R^{\I}(\vec \psi(t))+\int_R^{\infty} \psi(t)^2\, dr\right) \right] \, dt
 \end{align*}
 Using \eqref{small tail}, we fix $R$ large enough so that $$\sup_{t\ge0} O\left( \E_R^{\I}(\vec \psi(t))+\int_R^{\infty} \psi(t)^2\, dr\right) < \frac{ \E(\vec \psi)}{360}$$ Therefore, we deduce that 
 \begin{align}\label{integrated ineq}
   \ang{\chi_R \dot{\psi}\, \vert \, r\psi_r +  \frac{29}{20}  \psi}\Big|_0^T   &\le -\frac{1}{360} \E(\vec \psi) T 
  \end{align}
  for every $T>0$. However, we can use Hardy's inequality and the conservation of energy to estimate the left hand side of the above inequality as follows,
   \begin{align*}
  \abs{\ang{\chi_R \dot{\psi}\, \vert \, r\psi_r +  \frac{29}{20}  \psi} }  &\le  \abs{\int_1^{\infty}  \chi_R \dot \psi \psi_r\, r^3 \,dr} + C \abs{ \int_1^{\infty} \chi_R \dot \psi \psi \, r^2 \, dr}\\
  \\
  &\lec R \int_1^{\I} (\dot \psi^2 + \psi_r^2 + \frac{\psi^2}{r^2}) \, r^2 \, dr   \\
  \\
  & \lec R \E(\vec \psi)
  \end{align*}
  Combining the above with~\eqref{integrated ineq} we conclude that  
  \begin{align*}
  T\frac{1}{360} \E(\vec \psi)  \lec R\, \E(\vec \psi) 
  \end{align*}
  for all $T>0$, which, since $\E(\vec \psi)=\textrm{const}$,  implies that $ T \le C R  $. And this contradicts the fact that $\vec\psi$ exists globally in time. 
 \end{proof}

We can now complete the proof of Theorem~\ref{main}.

\begin{proof}[Proof of Theorem~\ref{main}] Suppose that Theorem~\ref{main} fails. Then Proposition~\ref{prop:3} implies the existence of a critical element, i.e., a nonzero energy class solution $\vec \psi(t) \in \HH$ to \eqref{WM} such that the trajectory $K_{+} = \{\vec \psi(t) \vert t\ge 0\}$ is pre-compact in $\HH$. However, Proposition~\ref{rigid} implies that any such solution must be identically zero, which contradicts the fact that the critical element is nonzero. 
\end{proof}

\subsection{Proof of Lemma~\ref{coercive}}
The remaining piece of the argument is the proof of Lemma~\ref{coercive}. To begin we define $\La: \dot H_0^1(1, \infty) \to \R$ by
\begin{align}\label{La}
 \La(\psi):= -\frac{9}{10} \int_1^{\I} \psi_r^2 \, r^2 \, dr  + \int_1^{\infty} \left( \sin^2(\psi) - \frac{29}{20} \psi \sin(2 \psi) \right) \, dr 
  \end{align}
And we note that in order to prove Lemma~\ref{coercive}, it suffices to show that 
\begin{align}\label{LA le 0} 
 \La( \psi) \le 0 \quad \textrm{for every} \quad \psi \in \dot{H}^1_0(1, \infty)
 \end{align}
Indeed, if \eqref{LA le 0} holds then 
\begin{align*}
\LL(\vec{\psi}) &=   -\frac{1}{20} \int_1^{\infty} \left( \dot \psi^2 + \psi_r^2 \right) r^2 \, dr + \La(\psi) \\
\\
& \le  -\frac{1}{20} \int_1^{\infty} \left( \dot \psi^2 + \psi_r^2 \right) r^2 \, dr
\end{align*}
which is exactly  \eqref{coerce}. 
For each $R>1$, define 
\begin{align*}
\A_R:= \{\psi \in \dot H^1_0(1, \infty) \mid \psi(r) = 0\; \textrm{for every} \; r \ge R\}
\end{align*}
Observe that $\A_R = \dot H^1_0(1,R)$ where the subscript $0$ indicates  Dirichlet boundary conditions at both $r=1$ and $r=R$. We start by deducing \eqref{LA le 0} on $\A_R$ for each $R>1$. 

\begin{lem}\label{A_R} For each $R>1$  the restriction $\La\vert_{A_R}: \A_R \to \R$ satisfies  $\La(\psi) \le 0$ for every $\psi \in \A_R$.  

\end{lem}
Assuming Lemma \ref{A_R}, we can extend  \eqref{LA le 0} to all of  $\dot{H}_0^1(1, \infty)$ via an approximation argument as follows. To simplify notation, set
\begin{align*}
F(\psi) &:=  \sin^2(\psi) - \frac{29}{20} \psi \sin(2 \psi)\\
N(\psi)&:=  \int_1^{\infty} F(\psi(r)) \, dr\\
E(\psi)&:= \frac{1}{2}\int_1^{\infty} \psi_r^2(r) \, r^2 \, dr
\end{align*}
Then, 
\begin{align*}
\La(\psi) = -\frac{9}{5}E(\psi) + N(\psi) 
\end{align*}

\begin{proof}[Proof that Lemma \ref{A_R} implies Lemma \ref{coercive}]
Assume that Lemma \ref{A_R} is true but \eqref{LA le 0} fails. Then there exists $\psi \in \dot{H}_0^1(1, \infty)$ such that 
\begin{align} \label{fail}
\La( \psi) = \de>0
\end{align}
For each $k \in \N $ define $\phi_k \in C_0^{\infty}(\R) $ so that $\phi_k(r) = 1$ for $0 \le r \le k$, $\phi_k \equiv 0$ for $r \ge 2k$ and $\abs{\phi'_k(r)} \lesssim \frac{1}{k}$. Then set $\psi_k:= \phi_k \psi$. Note that for each $k$,  $\psi_k \in \A_{2k}$ and that 
\begin{align*} 
E(\psi_k) &\to E(\psi) \; \textrm{as}\;  k \to \infty\\
N( \psi_k) & \to N(\psi)  \; \textrm{as}\;  k \to \infty
\end{align*}
Hence, by \eqref{fail}, there exists $k_0 \in \N$ such that   
\begin{align*}
\La(\psi_k) \ge \frac{\de}{2} >0
\end{align*} 
for $k \ge k_0$, and this contradicts Lemma \ref{A_R}. 
\end{proof}

Therefore,  it remains to establish Lemma \ref{A_R}.  In what follows we
fix $R>1$. The goal is to show via a variational argument that $\psi\equiv 0$ maximizes $\La\vert_{\A_R}$. Since $\La(0) =0$, this would prove Lemma \ref{A_R}.   

We claim that $\La$ defines a bounded functional on $\A_R$. To see this, observe that for every $x$, we have $\abs{F(x)} \le 2\abs{x}$.  Hence by the Strauss estimate, \eqref{strauss}, and the fact that we are in $\A_{R}$, we have
\begin{align*}
N(\psi) \le 2 \int_1^R \abs{ \psi(r)} \, dr \le 8R \sqrt{E(\psi)}
\end{align*}
Therefore,
\begin{align}\label{La bound}
\La(\psi) \le -\frac{9}{5}E(\psi) + 8R \sqrt{E(\psi)} \le C(R)
\end{align} 
Since $\La$ is bounded on $\A_R$ and $\La(0)=0$, we define $0\le \mu \le C(R)$ by 
\begin{align*}
\mu:= \sup_{\psi \in \A_R}\La(\psi)
\end{align*}
Now, let $\{\psi_n\}_{n=1}^{\infty} \subset \A_R$ be a maximizing sequence, i.e., $\La(\psi_n) \to \mu$ as $n \to \infty$. We claim that $E(\psi_n) \le C$. If not, then there exists a subsequence, $\{\psi_{n_k}\}$ such that $E(\psi_{n_k}) \to  \infty$. But then, by~\eqref{La bound}, we would have $\La(\psi_{n_k}) \to -\infty$, which contradicts the fact that $\{\psi_n\}$ is maximizing and $\mu\ge 0$.  Since $E(\psi_n)= \frac{1}{2}\|\psi_n\|_{\dot H^1}^2 \le C$ we can extract a subsequence, still denoted by $\{\psi_n\}$, so that 
\begin{align*}
\psi_n &\weakto \psi_{\infty} \in \dot H^1_0\\
\psi_n & \to \psi_{\infty} \in L^2_{\textrm{loc}}\\
\psi_n & \to \psi_{\infty} \; \textrm{pointwise a.e. on}\; [1, R]
\end{align*}
And, since $\A_R = \dot{H}^1_0(1, R)$, the boundary conditions are automatically satisfied and we have $\psi_{\infty} \in \A_R$. Next, we claim that $\psi_{\infty}$ is in fact a maximizer, i.e., $\La(\psi_{\infty})= \mu$. On the one hand, since $\mu$ is the supremum,  $\La(\psi_{\infty}) \le \mu$.  
To prove the other direction we remark that by the lower semi-continuity of weak limits we have that
\begin{align*}
\ds{\liminf_n E(\psi_n) \ge E(\psi_{\infty})}
\end{align*}
Also,  since $\abs{F(\psi_n)} \le  3 \psi_n^2 \le 6 E(\psi_n) \le C$, by the bounded convergence theorem, we see that
 \begin{align*}
 \lim_{n\to \infty} N(\psi_n) = N(\psi_{\infty}). 
\end{align*}
Putting this together we get 
\begin{align*}
\La(\psi_{\infty}) - \mu &= \lim_{n \to \I} (\La(\psi_{\infty}) - \La(\psi_n))\\
&= \lim_{n \to \I} \left(-\frac{9}{5}E(\psi_{\infty}) +\frac{9}{5}E(\psi_n) + N(\psi_{\infty}) - N(\psi_n)\right)\\
&\ge \frac{9}{5}\liminf_{n \to \I}(-E(\psi_{\infty}) +E(\psi_n) ) + \liminf_{n \to \I}( N(\psi_{\infty}) - N(\psi_n))\\
&\ge \liminf_{n \to \I}( N(\psi_{\infty}) - N(\psi_n)) =0
\end{align*} 
Hence $\La(\psi_{\infty})= \mu$ and so $\psi:=\psi_{\infty} \in \A_{R}$ is our maximizer. Now, let $\eta \in C^{\infty}_0(1, R)$ and  consider compact variations  $\psi_{\e} := \psi + \e \eta$ of $\psi$.  Since $\psi$ is a maximizer for $\La\vert_{\A_R}$, it follows that 
\begin{align*}
0 = \frac{d}{d \e} \La(\psi_{\e})\vert_{\e=0}&= - \frac{9}{5}  \int_1^{\I}  \psi_r \eta_r \, r^2 \, dr + \int_1^{\infty} F'(\psi) \eta \, dr\\
\\
&=   \int_1^{\I}  \left( \frac{9}{5} r^{-2} \p_r (r^2 \psi_r) + \frac{F'(\psi)}{r^2}\right) \, \eta\,   r^2\, dr\
\end{align*}
This implies that  $\psi$ satisfies the following Euler-Lagrange equation
\begin{align}\label{EL}
\psi_{rr} + \frac{2}{r} \psi_r &= -\frac{5}{9}\frac{F'(\psi)}{r^2}\\ \notag
\psi(1) &=0, \, 
\psi(R)=0
\end{align}
where the boundary conditions originate with the requirement that $\psi \in \A_R$. Setting $r= e^x$ and defining $\fy(x):= \psi(e^x)$ we obtain the following autonomous differential equation for $\fy$: 
\begin{align}\label{fy}
\fy{''} + \fy' &= f(\fy)\\ \notag
\fy(0)&=0 , \, 
\fy(\log(R))=0
\end{align}
where $f(\fy):= -\frac{5}{9}F'(\fy)= \frac{1}{4}\sin(2\fy) + \frac{29}{18} \fy \cos(2\fy)$. We claim that $\fy \equiv 0$ is the only solution to \eqref{fy}. Note that this implies Lemma \ref{A_R} since then $\psi\equiv 0$ would be the unique maximizer for $\La\vert_{\A_R}$ and $\La(0)=0$.  We formulate the claim as a general lemma about the differential equation \eqref{fy}. 

\begin{lem}\label{ode lem} Let $f(x):= \frac{1}{4}\sin(2x) + \frac{29}{18} x \cos(2x)$.  Suppose that $x(t)$ is a solution to 
\begin{align}\label{ode}
\ddot{x} + \dot x &= f(x)
\end{align}
and suppose that $x(0)=0$ and that there exists a $T>0$ such that $x(T)=0$. Then $x \equiv 0$. 
\end{lem}
We note that the conclusion of Lemma~\ref{ode lem} depends highly on the exact form the function~$f$. In fact, the lemma fails if we replace $f$ with $ \frac{3}{2} f$. Such a change would amount to requiring a smaller fraction of $E(\psi)$ to dominate $N(\psi)$ in \eqref{LA le 0}.  This subtlety necessitates the careful analysis that is carried out in the proof.

The proof of Lemma \ref{ode lem} will consist of a detailed analysis of the phase portrait associated to \eqref{ode}.   Letting $y(t):=\dot x(t)$, and setting 
\begin{align*}
v(t)&:= (x(t), y(t))^{\textrm{tr}}  \\ N(x,y&): = (y, -y +f(x))^{\textrm{tr}}
\end{align*}
 we rewrite \eqref{ode} as the following system 
\begin{align}\label{s ode}
\dot v&:=\begin{pmatrix} \dot x \\ \dot y\end{pmatrix} = \begin{pmatrix} y \\ - y + f(x) \end{pmatrix}=: N(v)\\ \notag
\end{align}
We can make a few immediate observations about the behavior of solutions to \eqref{s ode}. First we note that since $\abs{N(v)} \le C\abs{v}$, Gronwall's inequality implies that  solutions are unique and exist globally in time. Let $\Phi_t$ denote the flow.

Next observe that  equilibria of~\eqref{s ode} are all{ \em hyperbolic} (following the terminology of Wiggins~\cite{Wiggins}) and that they occur at the points $v_j:=(x_j, 0)$, where $x_j$ is a zero of $f$, i.e., $f(x_j) =0$. To see this we linearize about the equilibrium $v_j$, which results in the the equation 
\begin{align}\label{lin}
 \dot \xi = \na N(v_j) \xi
 \end{align} 
where 
\begin{align*}
\na N(v_j)= \begin{pmatrix} 0 & 1\\ f'(x_j) & -1 \end{pmatrix}
\end{align*}
The eigenvalues of $\na N (v_j)$ are given by 
\begin{align}\label{eigen}
\la_{\pm}(v_j) =  -\frac{1}{2} \pm \frac{1}{2}\sqrt{1+ 4f'(x_j)} 
\end{align}
To proceed, a more careful examination of the zeros of $f$ is required. We  can order the zeros $x_j$ so that 
\begin{align*}
\dots  x_{-j} <\dots <x_{-1}<0=:x_0 < x_1< \dots <x_j \dots
 \end{align*}
  We note that since $f$ is odd one has $x_{-j}=-x_j$ and it suffices to look at only those $x_j$ such that $x_j \ge0$. Indeed, all properties of the phase portrait on the right-half plane are identical 
  to those on the left-half plane after a reflection about the origin. 

 First, observe that $x_0:=0$ satisfies $f(x_0)=0$ and $f'(x_0) = \frac{19}{9} >2$. Hence,  $\la_{+}(v_0) >-\frac{1}{2} + \frac{3}{2} = 1 >0$ and $\la_{-}(v_0) < -\frac{1}{2}$. This means that  \eqref{s ode} has a {\em saddle} at $v_0=(0, 0)$.  Next, we see  that  due to the oscillatory nature of $f$ and the fact that $f'(0)>0$  we can deduce that $f'(x_j) >0$ for $j$ even, and  $f'(x_j) <0$ for $j$ odd. It is also straightforward to show that $\abs{f'(x_j)} >1$ for every $j>0$. These facts, together with \eqref{eigen} imply that 
 \begin{align*}
 &\Re (\la_{\pm}(v_j) )< 0 \; \textrm{if} \; j\; \textrm{is odd} \\
&   \la_{+}(v_j) >0,\; \textrm{and} \;  \la_{-}(v_j)<0\; \textrm{if} \; j\; \textrm{is even}
 \end{align*} 
Hence \eqref{s ode} has {\em sinks} at each $x_j$ for $j$ even, and {\em saddles} at each $x_j$ for $j$ odd. Also we note that in a neighborhood $V_j \ni v_j$, the equilibira $v_j$,  for $j$ even, each have a $1$-dimensional invariant stable manifold 
\begin{align*}
W_j^s := \{v \in V_j \;\vert\;  \Phi_t(v) \in V_j\;  \forall\; t \ge 0, \;  \Phi_t(v) \to v_j \; \textrm{exponentially as} \; t \to + \infty\}
\end{align*}
and a $1$-dimensional invariant unstable manifold 
\begin{align*}
W_j^u := \{v \in V_j \;\vert\;  \Phi_t(v) \in V_j\;  \forall\; t \le 0, \;  \Phi_t(v) \to v_j \; \textrm{exponentially as} \; t \to - \infty\}
\end{align*}
that are tangent to the respective invariant subspaces of the the linearized vector field corresponding to the right hand side of \eqref{lin} at the point $v_j$. For $j$ even, the stable invariant linear subspace at $v_j$ is  spanned by $\xi_{-}(v_j) = (1, \la_{-}(v_j))$  and the unstable invariant subspace is spanned by $\xi_{+}(v_j) = (1, \la_{+}(v_j))$.   The equilibria $v_j$, for $j$ odd, each have a two dimensional invariant stable manifold, (see, for example, \cite{NakS}, Chapt. 3.2). 

Our goal is to demonstrate the impossibility of a trajectory $v(t)$ such that $v(0)= (0, y_0)$  and $v(T) = (0, y_T)$  with $y_0 \neq 0$ and $T \in \R$. By symmetry considerations we can restrict ourselves to the case  $y_0 >0$. We rule out  such a trajectory by showing that solutions with data on the unstable invariant manifolds at the equilibria $v_j$, for $j$ even, have the following properties:

\begin{lem}\label{phase} 
Let $j=2\ell$ be even. Denote by $v^{+}_j=(x^{+}_j, y_j^{+})$  the unique trajectory with data in $W_j^u$ such that there exists a $\tau_1>0$ large enough so that $y_j^{+}(t) >0$ for all $t <-\tau_1$. And denote by $v^{-}_j=(x^{-}_j, y_j^{-})$  the unique trajectory in $W_j^u$ such that there exists a $\tau_2>0$ large enough so that $y_j^{-}(t) <0$ for all $t <-\tau_2$. Then, the following statements hold.

\begin{list}{(\roman{parts})}{\usecounter{parts}}
\item There exists $T_1 \in \R $ such that $v^{+}_j(T_1) = (p_j^{+}, 0)$ with $p_j^{+} \in (x_{j+1}, x_{j+2})$. \label{pos traj}

\item There exists $T_2 \in \R $ such that $v^{-}_j(T_2) = (p_j^{-}, 0)$ with $p_j^{-}\in (x_{j-2}, x_{j-1})$.\label{neg traj}
\end{list}
We assume that $T_1, T_2$ are minimal with the stated properties.
\end{lem}
 The conclusion of Lemma \ref{phase} is depicted in Figure~\ref{phaseportrait}.


\begin{figure}
\labellist
\endlabellist
\includegraphics[scale=.3]{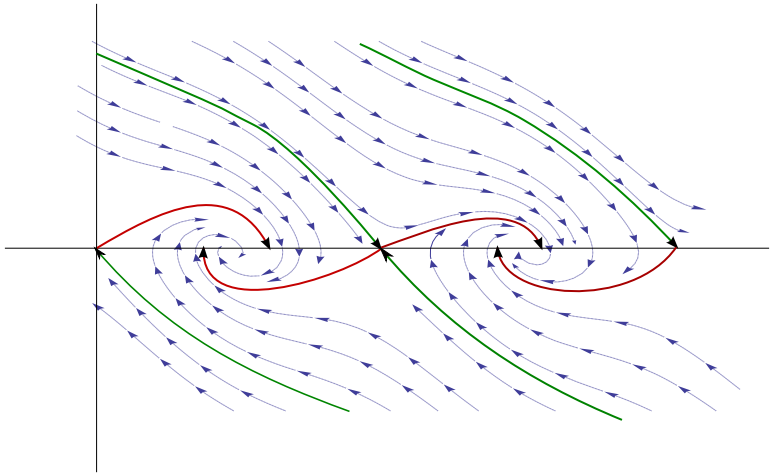} \caption{The figure above represents a slice of the phase portrait associated to \eqref{s ode}. The red flow lines represent the unstable manifolds, $W^{u}_j$, associated to the $v_j$, and the green flow lines represent the stable manifolds, $W^s_j$, associated to the $v_j$.  \label{phaseportrait} } 
\end{figure}


\begin{proof}[Proof that Lemma~\ref{phase} implies Lemma~\ref{ode lem}] Suppose we start with data $v(0)=(0, y_0)$ with $y_0>0$. Then, since the right hand side of~\eqref{s ode} is given by $(y, -y)^{\textrm{tr}}$ on the line $\{x=0\}$, the trajectory $v(t)$ enters the right-half plane in forward time. Note that $v(t)$ can never cross back into the left-half plane when $y(t)>0$ since the line $\{x=0, y>0\}$ is repulsive with respect the forward trajectory of $v$.  Hence, in order for there to be a time $T>0$ such that $v(T)=(0, y(T))$ the trajectory must first cross into the lower-half plane. However, $v(t)$ must then either lie in the stable manifold $W_j^s$ for some even $j$, or by Lemma~\ref{phase}~$(i)$ it crosses the $x$-axis between $x_{k}$ and $x_{k+1}$ for some $k$ odd. But then, if the latter occurs,  by Lemma~\ref{phase} $(ii)$, the flow must cross back into the the upper-half plane again at some point strictly between $x_{k-1} $ and $x_k$. If we track the trajectory further, $(i)$ and $(ii)$ will, in fact, force $v(t)$ into the sink at $x_k$, thus preventing it from ever reaching the $y$-axis. By the reflection symmetry of~\eqref{s ode}, the same logic works if we begin with data $v(0)=(0, y_0)$ with $y_0<0$. 
\end{proof}

To simplify the picture we begin by  dividing the phase plane into strips by defining $\Om_{j/2+1} = [x_j,  x_{j+2}] \times \R$ for $j \in 2\Z$. We first verify  Lemma~\ref{phase} in $\Om_1$ and in  $\Om_2$ and then we will renormalize~\eqref{s ode} in order to treat cases $(i)$ and $(ii)$ in $\Om_{\ell}$ for $\ell \ge 3$. 

\begin{proof}[Proof of Lemma~\ref{phase}  on  $\Om_1$ and $\Om_2$]
 The main tool in the proof of Lemma~\ref{phase} in $\Om_1$ and $\Om_2$ will be the following identity which is obtained by multiplying equation~\eqref{ode} by $\dot x$ and integrating from $t=t_0$ to $t=t_1$. 
\begin{align} 
\int_{t_0}^{t_1} \ddot x(s) \dot x(s) \, ds + \int_{t_0}^{t_1} \dot x(s)^2 \, ds = \int_{t_0}^{t_1} f(x(s)) \dot x(s) \, ds
\end{align}
Substituting $y= \dot x$  this becomes
\begin{align} \label{cons} 
\frac{1}{2} (y^2(t_1) - y^2(t_0)) + \int_{t_0}^{t_1} y^2(s) \, ds = F(x(t_1)) - F(x(t_0))
\end{align}
where $F(x):= \frac{5}{18} \cos(2x) + \frac{29}{36} x \sin(2x)$ is a primitive for $f$. 


\begin{figure}
\labellist
\endlabellist
\includegraphics[scale=.4]{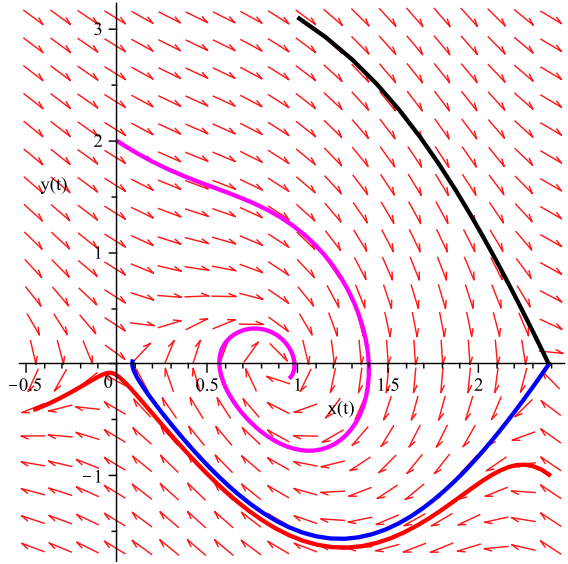} \caption{A schematic depiction of the flow in the first strip $\Omega_1$. \label{Om1}} 
\end{figure}


We will also need to approximate the zeros $x_0, x_1, \dots, x_4$. We can do this to any degree of precision, although a rather rough approximation will suffice. By inspection, the zero $x_j$ is close to the point $\frac{2j-1}{4} \pi$ for $j \ge1$. Indeed we have, 
\begin{align}
x_0=0 , \; x_1 \approx 0.8733, \; x_2 \approx 2.3886, \; x_3 \approx 3.9466, \; x_4 \approx 5.51186
\end{align}

First we show $(i)$ on  $\Om_1$. We would like to show that there  exists $T \in (-\infty, \infty] $ and $p \in [x_1, x_2]$ so that $v^{+}_0(T)= (p, 0)$. In the process we will also show that $x^{+}_0(t) \le x_{j+2}$ for all $t \in \R$. 

 Note that on the line $\{x=x_j\}$ in the phase plane the right-hand side of~\eqref{s ode} is equal to $(y, -y)^{\mathrm{tr}}$. Hence, the trajectory $v_0^{+}(t)$ can never enter the left-half plane $\{x<0\}$  by crossing the line $\{x=0, y >0\}$ as the vector field $(y, -y)^{\mathrm{tr}}$ is repulsive along this line in forward time. Also, since $\abs{f(x)} \le 3$ on $[0, x_2]$ the vector field $(y, -y+f(x))^{\textrm{tr}}$ prevents $v_0^{+}(t)$ from ever crossing above the line  segment $\{0 \le x \le x_2, y=4\}$. Similarly, $v_0^{+}(t)$ can never cross from the upper into the lower-half plane through the line segment $\{ 0 < x < x_1, y=0\}$,  since $f(x)>0$ on $(0, x_1)$ and thus the vector field $(0, f(x))^{\textrm{tr}}$ repulses such a trajectory in forward time. 
 
 Therefore, the only remaining possibilities for the forward trajectory $v_0^{+}(t)$ are for Lemma~\ref{phase}~$(i)$ to hold, or for one of the following two scenarios to occur: the trajectory  crosses the line $\{x=x_2, y>0\}$ in finite time, or it is heteroclinic connecting the saddles $(x_0,0)$ and~$(x_2,0$).  Suppose that either of the latter two cases occurs. Then, there exists  $T \in \R\cup\{\I\}$ such that $v_0^{+}(T)= (x_2, y(T))$ with $y(T) \ge 0$. But then, letting $t_0\to -\infty$ in \eqref{cons} we would have 
  \begin{align*} 
 \frac{1}{2} y^2(T) + \int_{\infty}^{T} y^2(s) \, ds = F(x_2) - F(0) \approx -2.1799<-2
\end{align*}
which is a contradiction since the left hand side is strictly positive. This proves $(i)$ for $\Om_1$. The proof of $(i)$ for $\Om_2$ is identical. One first shows that the only possibilities for the trajectory $v_2^{+}(t)$ are for either $(i)$ to hold, or for it to cross the line $\{x=x_4,\,  y>0\}$ in finite time, or to be to heteroclinic. And the latter two scenarios are impossible by~\eqref{cons} since then there would be a $T \in \R\cup\{\I\}$ so that 
\begin{align*} 
\frac{1}{2} y^2(T) + \int_{\infty}^{T} y^2(s) \, ds = F(x_4) - F(x_2) \approx -2.52841 < -2
\end{align*}
which contradicts the positivity of the left-hand-side above. 

We will also use~\eqref{cons} to prove~$(ii)$, although we will not get by as easily as in the proof of $(i)$, as we will need to estimate the size of the left hand side of~\eqref{cons} to obtain a contradiction. This will be achieved via the construction of a Lyapunov functional. Unfortunately, this is somewhat delicate as can been seen by means of the blue line in Figure~\ref{Om1} which is the unstable manifold $W_2^u$
as computed by Maple. While it does visibly fall into the sink, it does so much less dramatically than $W_0^u$. 
For $(ii)$, the relevant trajectory in $\Om_1$ is $v_2^{-}(t)$ which has data $v_2^{-}( -\infty) = (x_2, 0)$ and satisfies $y^{-}_2(t) <0$ for $t \le -\tau_2$. 
By symmetry, we can instead consider the trajectory $v_{-2}^{+}(t)$ in $W^u_{-2}$ so that $y^{+}_{-2}(t) > 0$ for $t< -\tau$. This trajectory lies in $\Om_{-1}$. 

Again one shows 
 that either  $(ii)$ holds, or the forward trajectory $v_{-2}^{+}(t)$ reaches the line $\{ x=0, y \ge 0\}$ in finite or infinite positive time. In order to arrive at a contradiction, we assume that the latter occurs. That is, we assume that there exists $T\in \R\cup\{\I\}$ such that $v_{-2}^{+}(T)= (0, y^{+}_{-2}(T))$ with $y^{+}_{-2}(T) \ge 0$. In this case we are able to use the attractive nature of the fixed point $(x_{-1}, 0)$ to construct a subset $\Sigma \subset \Om_{-1}$ so that the flow $v_{-2}^{+}(t)$ cannot enter $\Sigma$. In other words, the boundary of $\Sigma$ will be repulsive with respect to the forward trajectory of $v_{-2}^{+}$.

To construct $\Sigma$, we define three polynomials. First define $p_1$ as a function of $x$:
\begin{align*}
p_1(x) &:=  -\frac{3}{1000} + \frac{110}{47}\left( x + \frac{43}{18}\right) - \frac{89}{222} \left(x + \frac{43}{18}\right)^2 - \frac{23}{42} \left(x + \frac{43}{18}\right)^3  \\ \notag
& \quad + \frac{7}{85} \left(x + \frac{43}{18}\right)^4 + \frac{8}{303} \left(x + \frac{43}{18}\right)^5 - \frac{1}{446}\left(x + \frac{43}{18}\right)^6 - 
 \frac{1}{760}\left(x + \frac{43}{18}\right)^7 \\
 &\quad + \frac{1}{4035} \left(x + \frac{43}{18}\right)^8 - \frac{1}{13999}\left (x + \frac{43}{18}\right)^9\\ \notag
 \end{align*}
 Then, define $p_2$ and $p_3$ as functions of $y$ as follows: 
 \begin{align*} 
 p_2(y)&:=-\frac{6627}{638000} - \frac{17913}{29000} y - \frac{19}{75}\left(y - \frac{21}{22}\right)^2 - \frac{17}{80}\left(y - \frac{21}{22}\right)^3 \\ \notag
 &\quad- \frac{29}{106}\left(y - \frac{21}{22}\right)^4 - \frac{36}{115} \left(y - \frac{21}{22}\right)^5 - \frac{9}{20}\left(y - \frac{21}{22} \right)^6 - \frac{19}{31}\left(y - \frac{21}{22}\right)^7\\ \notag
 &\quad - \frac{32}{ 35}\left(y - \frac{21}{22}\right)^8 - \frac{42}{31}\left(y - \frac{21}{22}\right)^9
\end{align*}
and
\begin{align*}
p_3(y)&:=  -\frac{104159}{877500} - \frac{9383}{19500}y - \frac{18}{113}\left(y - \frac{3}{5}\right)^2 + \frac{2}{365}\left(y - \frac{3}{5}\right)^3\\ \notag
& \quad- \frac{38}{291}\left(y - \frac{3}{5}\right)^4 + \frac{3}{50}\left(y - \frac{3}{5}\right)^5 - \frac{21}{158}\left(y - \frac{3}{5}\right)^6 + \frac{6}{71}\left(y - \frac{3}{5}\right)^7\\ \notag
&\quad - \frac{2}{15}\left(y - \frac{3}{5}\right)^8 + \frac{7}{82}\left(y - \frac{3}{5}\right)^9 - \frac{31}{278}\left(y -\frac{3}{5}\right)^{10} + \frac{6}{121}\left(y - \frac{3}{5}\right)^{11}
\end{align*}
Finally,  we set $\Sigma = \Sig_1 \cup \Sig_2 \cup \Sig_3$ where, 
 \begin{align*}
 \Sigma_1 &:= \left\{(x,y) \in \Om_{-1}\, \vert\, -\frac{43}{18} + \frac{3}{1000} < x < -\frac{3}{5}, \, 0<y< p_1\left(-\frac{3}{5}\right)\right\} \\ 
\Sig_2 &:= \left\{(x,y) \in \Om_{-1}\, \vert\, -\frac{3}{5} < x < p_2(y),\,  \frac{3}{5}<y< \frac{21}{22}\right\}\\
\Sig_3 &:=  \left\{(x,y) \in \Om_{-1}\, \vert\, -\frac{3}{5} < x < p_3(y),\,  0<y< \frac{3}{5}\right\}
\end{align*}
 The region $\Sigma$ is pictured in Figure~\ref{Lyap 1}.  A few words are required in order to explain how one goes about constructing the region $\Sig$, and in particular, about how one finds the functions $p_k$. To choose $p_1$, one begins by finding an approximate solution to \eqref{ode} with data slightly to the right of $x_{-2}$ via power series expansions. This approximate solution is then shifted downward by a small amount, here we take $\frac{3}{1000}$. As we will see below,  this downward shift ensures that the resulting function forms a curve that is, at least  initially, a Lyapunov functional in that it is repulsive with respect to the true trajectory emanating from~$x_{-2}$, i.e., the unstable manifold~$W_{-2}^u$. We then define $p_1$ by approximating the coefficients of the polynomial we found by rationals. We cease to use the graph of $p_1$ as the boundary of $\Sig$ when it ceases to possess the desired Lyapunov properties.  We then define $p_2$ and $p_3$ in similar fashions making sure that all of the respective graphs are eventually joined together by curves that are also Lyapunov. In the case of the segment joining the graph of $p_1$ and $p_2$ this is achieved with a vertical line as depicted in Figure \ref{Lyap 1}. For $p_2$ and $p_3$ the matching  is done with a horizontal line.

 
\begin{figure}
\labellist
\pinlabel $\Sig_1$ [t] at 400 200
\pinlabel $\Sig_2$ [t] at 662 195
\pinlabel $\Sig_3$ [t] at 690 90
\pinlabel ${x_{-2}}$ [t] at 40 0
\pinlabel $\frac{21}{22}$ [t] at 900 270
\pinlabel $\frac{3}{5}$ [t] at 900 170
\pinlabel $-\frac{3}{5}$ [r] at 630 -20
\pinlabel $0$ [t] at 890 5
\pinlabel $p_1(x)$ [t] at 321 500
\pinlabel $p_2(y)$ [t] at 720 270
\pinlabel $p_3(y)$ [t] at 810 130
\endlabellist
\includegraphics[scale=.3]{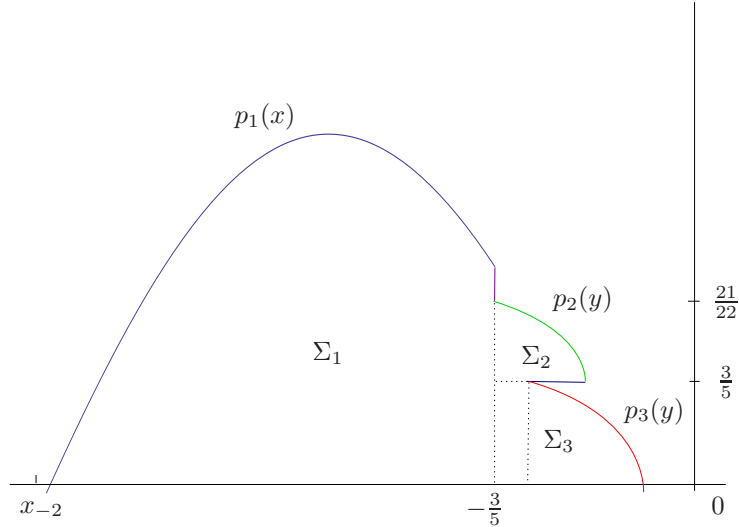} \caption{\label{Lyap 1} The region $\Sig= \Sig_1 \cup \Sig_2 \cup \Sig_3$ pictured above has the property that $\p \Sig$ is repulsive with respect to the unstable manifold $W_{-2}^u$.}
\end{figure}


We claim that the boundary of $\Sig$ is repulsive with respect to the trajectory $v_{-2}^{+}(t)$. To see this, it suffices to show that the outward normal $\nu$ on  $ \p \Sig \cap \{y>0\}$ satisfies 
\begin{align}\label{rep_1}
\nu \cdot N \ge 0
\end{align}
where $N:= (y, -y +f(x))^{\textrm{tr}}$ is the vector field \eqref{s ode}. There are five components to $\p \Sig \cap \{y >0\}$. Three components are given by the graphs of $p_1,\dots, p_3,$ and we label these components $\p\Sig_1, \dots,  \p \Sig_3$.   The other two components are given by the vertical segment, $\p \Sig_4$, connecting the point $(-\frac{3}{5}, \frac{21}{22})$ to $(-\frac{3}{5}, p_1(-\frac{3}{5}))$,  and the horizontal segment, $\p \Sig_5$, connecting the point $(p_3(\frac{3}{5}), \frac{3}{5})$ to $( p_2(\frac{3}{5}), \frac{3}{5})$. We must check that \eqref{rep_1} holds on each component. 

On $\p \Sig_1$ the outward normal $\nu_1$ is given by $ \nu_1= ( -p_1'(x), 1)$. On $\p \Sig_2$, $\nu_2 = (1, -p_2'(y))$. Similarly, $\nu_3 = ( 1, -p_3(y))$. Finally, $\nu_4=(1,0)$ and $\nu_5 = (0, -1)$.  And, it is elementary to check that indeed, 
\begin{align*} 
\nu_1 \cdot N &= f(x) - p_1(x)(1+p_1'(x))>0 \; \textrm{for every}\; -\frac{43}{18}\le x \le-\frac{3}{5}\\
\nu_2 \cdot N &= y+ p_2'(y)(y- f(p_2(y))>0   \; \textrm{for every}\; \frac{3}{5}<y\le\frac{21}{22} \\
\nu_3 \cdot N &=y+ p_3'(y)(y- f(p_3(y)) >0  \; \textrm{for every}\; 0\le y \le \frac{3}{5}
\end{align*}
as well as
\begin{align*}
\nu_4 \cdot N &=y >0 \; \textrm{for every}\; \frac{3}{5}\le y \le \frac{21}{22}\\
\nu_5 \cdot N &= \frac{3}{5} - f(x) >0 \; \textrm{for every}\; p_3\left(3/5\right)\le x \le p_2\left(3/5\right)\\
\end{align*}
Now, by \eqref{cons}, we have that 
\begin{align}\label{RHS 1}
 \frac{1}{2}y^2(T) + \int^T_{-\infty}y^2(s) \, ds = F(0)- F(x_{-2}) \approx 2.1799 <2.18
 \end{align}
 However, we claim that 
 \begin{align}\label{cont 1} 
 \int^T_{-\infty}y^2(s) \, ds > \textrm{Area}(\Sigma) > 2.18
 \end{align}
 To prove~\eqref{cont 1}, we first make the claim that under our current assumptions,  the integral on the left-hand side of~\eqref{cont 1} is greater than the area of the region bounded by the trajectory $v_{-2}^{+}(t)$ and the lines $\{x\le0\}$ and $\{y=0\}$. To see this recall that $v_{-2}^{+}(t)$ lies on the unstable manifold $W^u_{-2}$ and hence locally we can either write $y_{-2}^{+}(t)=y(x(t))$ or $x_{-2}^{+}(t) =x(y(t))$. Assume that for $\tau_0 < t < \tau_1$ we can write $y=y(x)$. Then, $x(\tau_0) < x(\tau_1)$ and
 \begin{align*} 
 \int_{\tau_0}^{\tau_1} y^2(s) \, ds = \int_{\tau_0}^{\tau_1} y(x(s))\dot{x}(s) \, ds= \int_{x(\tau_0)}^{x(\tau_1)} y(x) \, d x
 \end{align*}
 which, since $y(t)\ge0$, is, in fact, the area of the region bounded by the trajectory $v_{-2}^{+}(t)$, the line $\{y =0\}$, and the lines $\{x= x(\tau_0)\}$ and $\{x= x(\tau_1)\}$. 
 
 Next suppose we can write $x=x(y)$ for $\tau_2 < t < \tau_3$ and that $y(\tau_2) > y(\tau_3)$. Since all vertical lines in $\Om_{-1}$ have the property that they cannot be crossed by the flow from right to left in forward time we have that $x(y(\tau_2)) \le x(y(\tau_3))$. Observe that if $x=x(y(t))$ then $\dot x= x'(y) \dot y$, and hence 
  \begin{align*}
  \int_{\tau_2}^{\tau_3} \dot x(s)^2 \, ds &=  \int_{\tau_2}^{\tau_3}  y(s) x'(y(s)) \dot y(s) \, ds =  \int_{y(\tau_2)}^{y(\tau_3)} y\, x'(y) \, d y\\
  &= \int_{y(\tau_3)}^{y(\tau_2)} x(y) \, d y + y(\tau_3)x(y(\tau_3)) - y(\tau_2) x(y(\tau_2))
  \end{align*}
  but this can further be estimated from below by
  \begin{align*}
  &\ge \int_{y(\tau_3)}^{y(\tau_2)} x(y) \, d y  + (y(\tau_3) - y(\tau_2))x(y(\tau_2))\\
  &= \int_{y(\tau_3)}^{y(\tau_2)} \left[x(y)+ x(y(\tau_2))\right] \, d y
  \end{align*} 
 where the last line is exactly the area of the region bounded by $v_{-2}^{+}(t)$, and the lines $\{ x = x(\tau_2)\}$, and $\{ y =y(\tau_3)\}$.

Therefore,  since $v_{-2}^{+}(t)$ cannot enter $\Sigma$ we have  $ \int^T_{-\infty}y^2(s) \, ds > \textrm{Area}(\Sigma)$. The remaining step is to compute the area of $\Sig$ which can be done explicitly since $\Sig$ is defined entirely in terms of polynomials with rational coefficients. Indeed, 
\begin{align*}
\textrm{Area}(\Sig)=  \textrm{Area}(\Sig_1)+ \textrm{Area}(\Sig_2)+ \textrm{Area}(\Sig_3) >2.21
\end{align*}
which proves~\eqref{cont 1} and provides a contradiction when combined with~\eqref{RHS 1}. This proves $(ii)$ in $\Om_1$.  Note that small margin of error which is allowed here (after
all the relevant numbers are, respectively, $2.21$ and $2.18$) is a reflection of the ``almost heteroclinic" nature of the blue line in Figure~\ref{Om1} which is $W_2^u$. This forces us to be very precise
about the Lyapunov functionals that we constructed above. 

Next, we will establish $(ii)$ in $\Om_2$. The relevant trajectory is $v_{4}^{-}(t)$ which has data $v_4^{-}(-\infty)=(x_4, 0)$. As before,  we can show that the only possibilities for $v_4^{-}(t)$ are either that $(ii)$ holds, or that there exists a time $T \in \R\cup\{\I\}$ such that $v_4^{-}(T) = (x_2, y_4^{-}(T))$ where  $y_4^{-}(t)\le 0$ for all $-\infty <t \le T$. We assume the latter holds and seek a contradiction. As in the proof of $(ii)$ in  $\Om_1$ we will construct a subset $\Sig \subset \Om_2$ so that the boundary, $\p \Sig$, is repulsive with respect to the forward flow $v_4^{-}(t)$.  To construct $\Sig$ we define the polynomial 
\begin{align*}
p(x):= \frac{3}{100}+ \frac{15}{4} \left(x - \frac{11}{2}\right) + \frac{18}{89} \left(x - \frac{11}{2}\right)^2 - \frac{136}{181} \left(x - \frac{11}{2}\right)^3
\end{align*}
and define 
\begin{align*}
\Sig :=  \{(x,y) \in \Om_2 \, \vert\, 18/5 < x < 11/2 , \, p(x) < y< 0\}
\end{align*}
The function $p$ is constructed in the same fashion as the Lyapunov functional for $\Om_{-1}$ except that here we need  only a $3$rd order approximation. Indeed,  the trajectory $v_{4}^{-}$ is far from heteroclinic and thus provides us with a much larger margin for error as we seek a contradiction. 

\begin{figure}
\labellist
\endlabellist
\includegraphics[scale=.4]{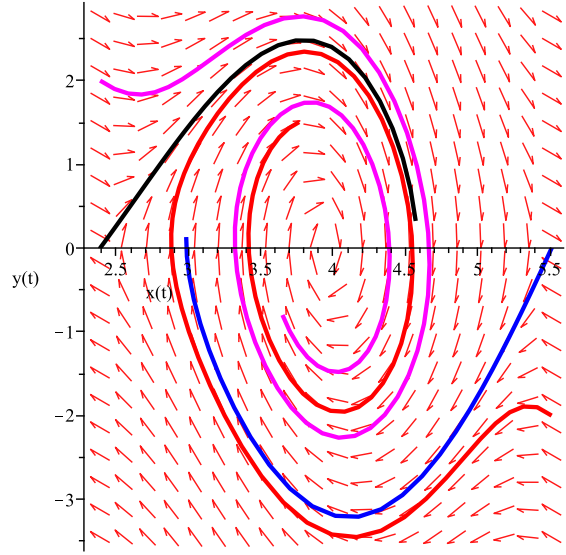} \caption{A schematic depiction of the flow in the second strip $\Omega_2$. \label{Om2}} 
\end{figure}


Again it suffices to show that the outward normal $\nu$ on $\p \Sig \cap \{y<0\}$ satisfies $\nu \cdot N \ge 0$.  We have $\nu = (p'(x), -1)^{\textrm{tr}}$. And one can  show that  
\begin{align*} 
\nu \cdot N = p(x)(1+ p'(x)) -f(x) > 0 \; \textrm{for every} \; 18/5 < x< 11/2
\end{align*}
 Again, we use \eqref{cons} to obtain,
  \begin{align}\label{RHS 2}
 \frac{1}{2}y^2(T) + \int^T_{-\infty}y^2(s) \, ds = F(x_2)- F(x_4) \approx 2.52841 <2.6
 \end{align}
However, we have 
\begin{align}\label{cont 2}
 \int^T_{-\infty}y^2(s) \, ds> \textrm{Area}(\Sigma) >3.8
 \end{align}
which contradicts  \eqref{RHS 2}. This completes the proof of Lemma~\ref{phase} in $\Om_1$ and in $\Om_2$.  We remark that the Lyapunov construction for~$\Om_2$
is considerably easier than for~$\Om_1$ as can be seen by Figure~\ref{Om2}. Indeed, the unstable manifold $W_4^u$, which is depicted by the blue trajectory in Figure~\ref{Om2},
is very far from being heteroclinic.   
\end{proof} 

To prove Lemma~\ref{phase} on $\Om_{\ell}$ for $\ell \ge 3$ we first shift and rescale~\eqref{s ode} via the following renormalization. For each $j\in \N$, $\e \in \R$  we define $\z$ and  $\eta$ via
\begin{align} \label{renorm}
x(t) &=:  \frac{2j-1}{4} \pi + \z(\e^{-1} t)\\ \notag
y(t)&=: \e^{-1}\: \eta(\e^{-1} t)
\end{align} 
Define $z_j:=\frac{2j-1}{4} \pi$. Then~\eqref{s ode} implies the following system of equations for $\z, \eta$
\begin{align} \label{renorm ode} 
\mat{ \dot \z\\ \dot \eta} = \mat{ \eta \\ -\e \eta + \e^{2} f(z_j + \z)}
\end{align}
where $\dot{\ }=\frac{d}{ds}$ where $s=\e^{-1} t$. 
 Observe that we have 
 \begin{align*}
 f(z_j + \z)=  (-1)^{j} \frac{29}{18}  z_j \sin(2 \z) + (-1)^{j+1} g(\z)
 \end{align*}
 where $g(\z):= \frac{1}{4} \cos(2\z) - \frac{29}{18} \z \sin(2 \z)$. Fix $j =2\ell$ with $\ell \ge 2$ and set 
 \begin{align}
 \e:=\sqrt{\frac{72}{29\pi(2j-1)}}
 \end{align}
 Note that $j \ge 4$ implies that $0<\e < \frac{7}{20}$. Then~\eqref{renorm ode} becomes
 \begin{align}\label{re ode}
 \mat{ \dot \z\\ \dot \eta} = \mat{ \eta \\    \sin(2 \z)  - \e \eta - \e^2 g(\z)}
 \end{align}
Note that~\eqref{re ode} is the equation governing the motion of a {\em damped pendulum} with a small perturbative term $\e^2 g(\z)$, and in the limit as $\e\to 0$,~\eqref{re ode} is exactly the the equation of a simple pendulum. 

Let's rephrase the set-up of Lemma~\ref{phase} in terms of this renormalization. First we examine how this affects the strip $\Om_{j/2 +1}$.  We can write the zeros of $f$ as
\begin{align*}
x_j&= z_j+ \z_0\\
x_{j+1} &= z_{j+1} + \z_1 = z_j + \frac{\pi}{2} + \z_1\\
x_{j+2}& =z_{j+2} + \z_2 = z_j + \pi + \z_2
\end{align*}
where $0< \z_0 <  \frac{\pi}{2} +\z_1<  \pi+\z_2$ are the first three positive zeros of $$h( \z):= \sin(2 \z) - \e^2 g(\z)$$ Hence the strip $\Om_{j/2+1}$ becomes the strip $\ti{\Om} = [\z_0,  \pi+ \z_2] \times \R$. Note that the renormalization~\eqref{renorm} does not affect the topological properties of the dynamics of~\eqref{s ode} and hence the invariant manifolds associated to the equilibria of \eqref{s ode} in $\Om_{j/2+1}$ become invariant manifolds associated to the equilibria of \eqref{re ode} in the strip $\ti \Om$.  Denote by $W_{\z_0}^u$ and $W_{\z_2}^u$, the unstable invariant manifolds associated to the equilibria $( \z_0, 0)$ and  $(\pi+ \z_2, 0)$. Thus Lemma~\ref{phase} in $\Om_{\ell}$ for $\ell \ge 3$ is equivalent to the following result. For simplicity, we again use~$t$ to denote time. 

\begin{lem}  \label{re phase} 
Denote by $v^{+}=(\z^{+}, \eta^{+})$  the unique solution of \eqref{re ode} with data in $W^u_{\z_0}$ such that there exists a $\tau_1>0$ large enough so that $\eta^{+}(t) >0$ for all $t <-\tau_1$. And denote by $v^{-}=(\z^{-}, \eta^{-})$  the unique solution in $W_{\z_2}^u$ such that there exists a $\tau_2>0$ large enough so that $\eta^{-}(t) <0$ for all $t <-\tau_2$. Then, the following statements hold:

\begin{list}{(\roman{parts})}{\usecounter{parts}}
\item There exists $T_1 \in \R $ such that $v^{+}(T_1) = (p_1, 0)$ with $p_1 \in (\pi/2 +\z_1 , \pi)$. 
\item There exists $T_2 \in \R  $ such that $v^{-}(T_2) = (p_2, 0)$ with $p_2\in(\z_0, \pi/2 + \z_1)$.
\end{list}
Again, we let $T_1, T_2$ be minimal with these properties. 
\end{lem}
 The proof of Lemma~\ref{re phase} will require a rather precise knowledge of the location of the zeros $ \z_0$ and $\pi+ \z_2$  of $h(\z)$. 

 \begin{lem}\label{zeros} Set $h(\z)=  \sin(2 \z) - \e^2 g(\z)$. Then \begin{list}{(\alph{parts})}{\usecounter{parts}}
 \item There exists a function $a: [0, \frac{7}{20}] \to [-\frac{1}{3}, -\frac{1}{9}] $ such that $h $ has a zero at  $\z_0= \z_0(\e) =  \frac{1}{2}\e^2g(0)( 1+ a(\e)\e^4)$. 
 \\
 \item There exists a function $c: [0, \frac{7}{20}] \to [10, 40]$ such that $h $ has a zero at $ \pi+ \z_2 = \pi + \z_2(\e)=  \pi+\frac{1}{2} \e^2g(\pi)( 1- \frac{29}{18} \pi \e^2 + c(\e) \e^4)$. 
 \end{list}
 In particular, $\zeta_0>0$ and $\zeta_2>0$. 
 \end{lem}
 We will postpone the proof of Lemma~\ref{zeros} for the time being and first establish Lemma~\ref{re phase}. 

\begin{proof}[Proof of Lemma~\ref{re phase}] Again our main tool will be the following identity, which is deduced in the same manner as~\eqref{cons},  
\begin{align}\label{re cons}
\frac{1}{2}(\eta^2(t_1) - \eta^2(t_0)) + \e \int_{t_0}^{t_1} \eta^2(s) \, ds &= \int_{t_0}^{t_1} \sin(2 \z) \dot \z \, ds -  \e^2 \int_{t_0}^{t_1} g(\z(s)) \dot \z(s)  \, ds \\ \notag
&=  \frac{1}{2}\left( \cos(2 \z(t_0)) - \cos( 2\z(t_1))\right)\\ \notag
& \quad - \e^2 (G(\z(t_1))- G(\z(t_0)))
\end{align}
where $G(x):=  \frac{29}{36} x\cos(2 x) - \frac{5}{18} \sin(2 x)$ is a primitive of~$g$. 

First we prove~$(i)$. The only possibilities for the forward trajectory $v^{+}(t)$ are for~$(i)$ to hold, or for there to exist a time $T$, possibly infinite, such that $v^{+}(T)= (\pi , \eta^{+}(T))$ with $0\le  \eta^{+}(t)$ for all $t\le T$. In this latter case,~\eqref{re cons} implies that 
\begin{align*}
\frac{1}{2} \eta^2(T) + \e \int_{-\infty}^T \eta^2(s) \, ds &= \frac{1}{2}\left( \cos(2 \z(t_0)) - 1)\right) - \e^2 (G(\pi)- G(\z_0))\\
& \le -\e^2 (G(\pi)- G(\z_0))\le 0
\end{align*}
which is a contradiction since the left-hand-side above is strictly positive.

Now, assume~$(ii)$ fails. Then there exists a time $T\in \R\cup\{\I\}$ such that $v_{-}(T) = (\z_0, \eta^{-}(T))$ with $ \eta^{-}(t) \le 0$ for every $t \le T$. As in the proof of Lemma~\ref{phase} $(ii)$ for $\Om_1$ and $\Om_2$ we construct a region $\Sig$ in $\ti \Om$ so that the boundary $\p \Sig$ is repulsive with respect to the flow $v^{-}(t)$. Set 
\begin{align}
y_1( \z) &:= -\frac{5}{4} \sin( \z) \\
y_2(\z)&= -\frac{5}{4} \sin(2) \sqrt{ 1- \frac{25}{36}(\z-2)^2}
\end{align}
Define $\Sig= \Sig_1 \cup \Sig_2$ by 
\begin{align} 
\Sig_1&:= \{(x, y) \in \ti \Om \, \vert \, 2\le x\le \pi, \, y_1(x) \le y\le 0 \}\\
\Sig_2&:=  \{(x, y) \in \ti \Om \, \vert \, \frac{7}{4}\le x< 2, \, y_2(x) \le y\le 0  \}
\end{align}
The region $\Sig$ is depicted in Figure~\ref{Lyaprenormal}.


\begin{figure}
\labellist
\pinlabel $\frac{7}{4}$ [t] at 20 550
\pinlabel $2$ [t] at 175 530
\pinlabel $\pi$ [t] at 855 530
\pinlabel $\Sig_1$ [t] at 350 300
\pinlabel $\Sig_2$ [t] at 100 300
\pinlabel $y_1(x)$ [t] at 600 230
\pinlabel $y_2(x)$ [t] at 100 10
\endlabellist
\includegraphics[scale=.3]{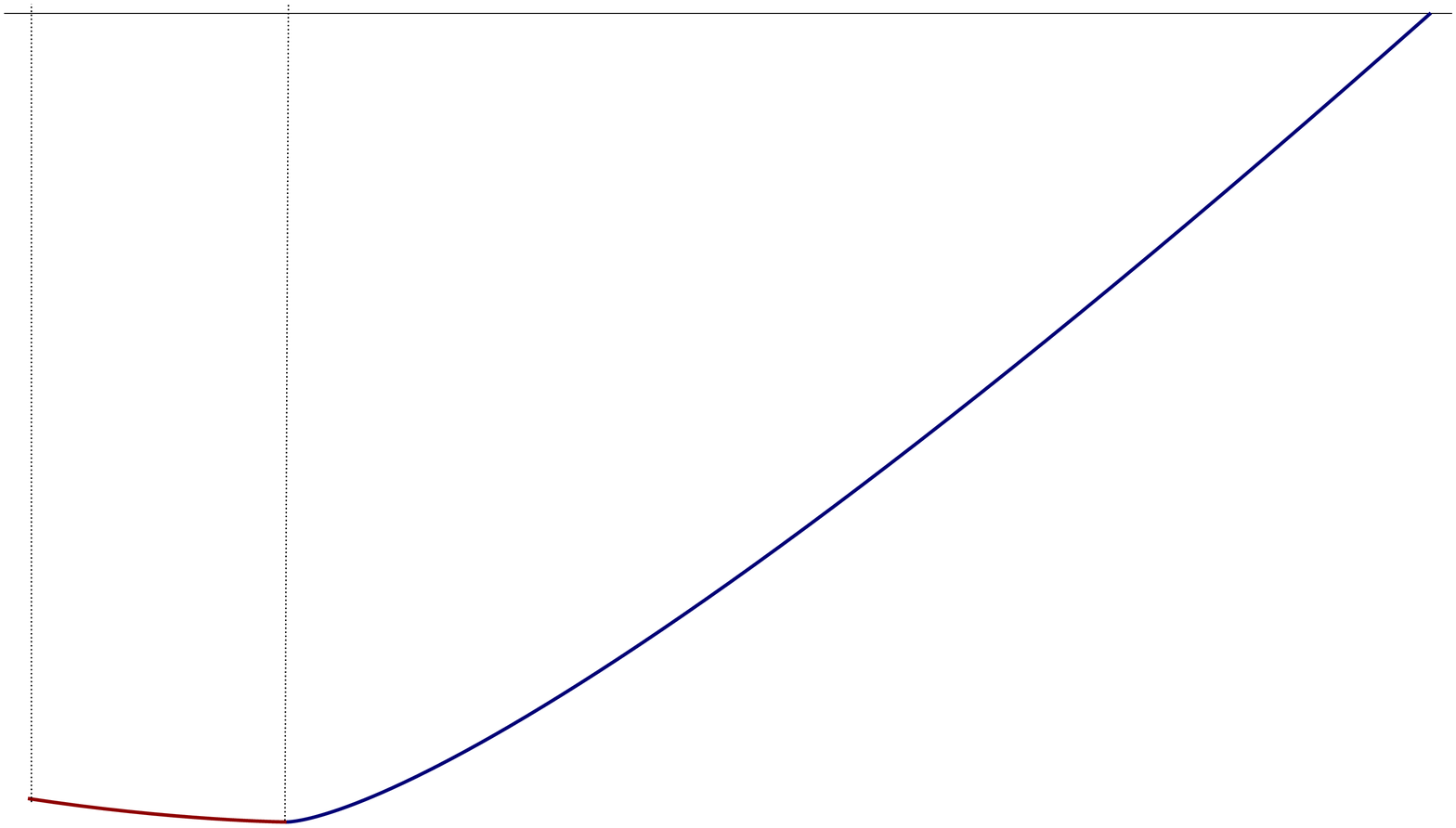} \caption{The region $\Sig= \Sig_1 \cup \Sig_2$ pictured above has the property that $\p \Sig$ is repulsive with respect to the unstable manifold $W_{\z_2}^u$.\label{Lyaprenormal}}
\end{figure}


Once again we need to check that the outward normal vectors $\nu_1$ on $\p \Sig_1$ and $\nu_2$ on $\p \Sig_2$ satisfy $\nu_k \cdot \ti N \ge 0$ for $k=1, 2$, where 
$$\ti N(\zeta,\eta) = ( \eta\, ,\, \sin(2 \z) - \e \eta - \e^2g(\z))^{\textrm{tr}}$$  Here $\nu_1= ( y_1'(\z), -1)^{\textrm{tr}}$ and $\nu_2 = (y_2'(x), -1)^{\textrm{tr}}$ and we have 
\begin{align}
 \nu_1 \cdot \ti N &=   -\frac{y_1(x)}{\be^2}F_1(x, \e)\label{lyap 1} \\
 \nu_2 \cdot \ti N& =  -\frac{y_2(x)}{\be^2 \sin^2(2)} F_2(x, \e)
 \end{align}
where, for $\al:=\frac{6}{5}$ and $\be:= \frac{5}{4}$, $F_1$ and $F_2$ are defined by
\begin{align} 
F_1(x, \e)&:= 2g(x) \e^2 -2\be\sin(x)\e + (\be^2-2) \sin(2x) \label{F_1}\\
 F_2(x, \e)&:=g(x) \e^2 - \e \frac{\be \sin(2)}{\al}  \sqrt{\al^2-(x-2)^2} \label{F_2}\\
 &\quad -\frac{\be^2 \sin^2(2)(x-2)+ \al \sin(2x)}{ \al^2}  \notag
\end{align}
Observe that  $y_1(x) \le 0$ for $2 \le x \le \pi$, and $y_2(x) \le 0$ for $\frac{7}{4} \le x \le 2$. Hence, the following lemma will suffice to conclude that  $\nu_k \cdot \ti N \ge 0$ for $k=1, 2$. 

\begin{lem}\label{F1 F2} Define $F_1, F_2$ as in  \eqref{F_1} and \eqref{F_2}. Then 
\begin{list} {(\roman{parts})}{\usecounter{parts}}
\item[$(A)$] $F_1(x, \e) \ge 0$ for every $(x, \e) \in [2, \pi] \times [0, \frac{7}{20}]$. 
\item[$(B)$] $F_2(x, \e) \ge 0$ for every $(x, \e) \in [\frac{7}{4}, 2] \times [0, \frac{7}{20}]$. 
\end{list}
\end{lem}
For the moment we assume Lemma~\ref{F1 F2} and observe that it implies that the boundary of $\Sig$ is repulsive with respect to the flow $v^{-}(t)$. By~\eqref{re cons} we have the following identity 
\begin{align} \label{cont 3}
\frac{1}{2}\eta^2(T)  + \e \int^{T}_{-\infty} \eta^2(s) \, ds &=  \frac{1}{2}\left( \cos(2\pi + 2\z_2) - \cos( 2\z_0)\right)\\ \notag
& \quad - \e^2 (G(\pi+ \z_2)- G(\z_0))
\end{align}
To arrive at a contradiction we carefully estimate the left and right-hand sides of~\eqref{cont 3}. By Lemma  \ref{zeros}, we can expand the right hand side in powers of $\e$. 
\begin{align} 
\frac{1}{2}\left( \cos( 2\z_2) - \cos( 2\z_0)\right) - \e^2 (G(\pi+ \z_2)- G(\z_0)) &= \frac{29 \pi}{36} \e^2 - \frac{29 \pi}{1152} \e^6 + O(\e^8)\\ \notag
& < \frac{29 \pi}{36} \e^2 \notag
\end{align}
for $0 \le \e \le \frac{7}{20}$. 

On the other hand, as in the proof of Lemma \ref{phase} for $\Om_1$ and $\Om_2$, we have that 
\begin{align} 
\e \int_{-\infty}^{T} \eta^2(s) \, ds &> \e \textrm{Area}( \Sig) = \e\left( -\int_{\frac{7}{4}}^{2} y_2(x) \, dx - \int_{2}^{\pi} y_1(x) \, dx\right) > \e
\end{align}
Finally,~\eqref{cont 3} then implies that $\e < \frac{29 \pi}{36} \e^2$ which is a contradiction for $0 \le \e \le \frac{7}{20}$. Hence, assuming the results of Lemma \ref{zeros} and Lemma~\ref{F1 F2}, we have established Lemma \ref{re phase} and therefore we have also completed the proof of Lemma~\ref{phase}.  
\end{proof} 

It remains to prove Lemma~\ref{zeros} and Lemma~\ref{F1 F2}. 

\begin{proof}[Proof of Lemma~\ref{zeros}] 
For fixed $a$, we plug $\z_0(a, \e) =  \frac{1}{2}\e^2g(0)( 1+ a\e^4)$ into $h$ and expand in powers of $\e$ about $\e=0$. This gives 
\begin{align*} 
h(\z_0(a, \e))= \left(\frac{1}{18} + \frac{a}{4}\right) \e^6 + O(\e^{10})
\end{align*}
With this in mind we set $a_1= -\frac{1}{3}$,  and obtain 
\begin{align}
h(\z_0(-\frac{1}{3}, \e)) = -\frac{1}{36} \e^6 + R_9( \e) 
\end{align}
where $R_9(\e) $ is the ninth remainder term in Taylor's theorem. One can show that for $0 \le \e \le \frac{7}{20}$, we have 
\begin{align*}
\abs{R_9(\e)} \le \sup_{0\le \abs{\xi} \le \e} \abs{\left(\frac{d}{d \xi}\right)^{10} h(\z_0( -\frac{1}{3} , \xi))} (10 !)^{-1} \e^{10} \le \e^{10}
\end{align*}
Hence, 
\begin{align*}
h(\z_0(-\frac{1}{3}, \e)) \le -\frac{1}{36} \e^6 + \e^{10}  \le - \frac{1}{36} \e^6 + \left(\frac{7}{20}\right)^4\e^6 \le 0
\end{align*} 
as long as $0 \le \e \le \frac{7}{20}$. Next we set $a = -\frac{1}{9}$ and we obtain
\begin{align*}
h(\z_0(-\frac{1}{9}, \e)) = \frac{1}{36} \e^6 + R_9( \e) 
\end{align*}
Again, one can show that $\abs{R_{9}(\e)} \le \e^{10}$   for $0 \le \e \le \frac{7}{20}$ and hence 
\begin{align*}
h(\z_0(-\frac{1}{9}, \e)) \ge \frac{1}{36} \e^6  - \e^{10} \ge   \frac{1}{36} \e^6  - \left( \frac{7}{20}\right)^4 \e^{6} \ge 0
\end{align*}
for $0 \le \e \le \frac{7}{20}$. This proves $(a)$. We carry out the same procedure to prove $(b)$. First,  fix $c$ and plug
 $ \pi + \z_2(c, \e)= \pi+  \frac{1}{2} \e^2g(\pi)( 1- \frac{29}{18} \pi \e^2 + c \e^4)$ into $h$ and expand in powers of $\e$ about $\e=0$. This gives,
  \begin{align*} 
 h( \pi +  \z_2(c, \e)) = \frac{(72 + 324 c - 841 \pi^2)}{ 1296}\e^6 + O(\e^8)
 \end{align*}
 Now, fix $c=10$. Then 
 \begin{align*}
  h( \pi +  \z_2(10, \e)) =\left(\frac{23}{9} -\frac{ 841 \pi^2}{1296}\right) \e^6 + R_7(\e)
\end{align*}
One can show that  $\abs{R_{7}(\e)} \le 20 \e^{8}$   for $0 \le \e \le \frac{7}{20}$ , and hence 
\begin{align*}
  h( \pi +  \z_2(10, \e)) \le -3.8 \e^6 + 20 \e^8 \le -3.8\e^6 + 2.5 \e^6 \le 0
 \end{align*}
as long as $0\le \e\le \frac{7}{20}$. Finally, set $c=40$. Then 
\begin{align*}
 h( \pi +  \z_2(40, \e)) =\left(\frac{181}{18} -\frac{841 \pi^2}{1296}\right) \e^6 + R_7(\e)
\end{align*}
One can show that  $\abs{R_{7}(\e)} \le 60 \e^{8}$   for $0 \le \e \le \frac{1}{8}$ , and hence 
\begin{align*}
  h( \pi +  \z_2(40, \e)) \ge 3.6 \e^6 - 60 \e^8 \ge 3.6\e^6 - \e^6 \ge 0
 \end{align*}
as long as $0\le \e\le \frac{1}{8}$. To conclude, we note that the positivity of $h( \pi +  \z_2(40, \e))$ on the compact interval $\e \in [\frac{1}{8}, \frac{7}{20}]$ is readily checked. 
\end{proof}

\begin{proof}[Proof of Lemma \ref{F1 F2}]Observe that for fixed, $x$,  $F_1(x, \e)$ and $F_2(x, \e)$ are quadratic functions in $\e$ and hence have  real zeros for $\e \in[0, \frac{7}{20}]$ if and only if their associated discriminants are nonnegative. One can readily check that the discriminant associated to $F_1(x, \cdot)$ is negative for each $2 \le x \le \pi$. And the discriminant associated to $F_2(x, \cdot)$ is  negative for each $\frac{7}{4} \le x \le 2$. Therefore, by continuity, $F_1$ has a fixed sign on  $[2, \pi] \times [0, \frac{7}{20}]$ and $F_2$ has a fixed sign on $[\frac{7}{4}, 2] \times [0 , \frac{7}{20}]$. Hence checking the positivity of $F_1$ and $F_2$ on their respective domains reduces to checking that they are positive at a single point. And, for example $F_1(\frac{5}{2}, \frac{1}{4}) \approx 0.54>0$ and $F_2(\frac{15}{8}, \frac{1}{4}) \approx .41>0$. 
\end{proof}

This concludes the proofs of Lemmas~\ref{coercive}--\ref{re phase}.

\section{The higher topological classes}
\label{sec:boring}

\noindent In this section we prove Theorem~\ref{thm:2}. 
By~\cite{BCM} we know that for each integer $n\ge1$ there is a unique solution $Q=Q_{n}$ to the stationary problem
\EQ{
-Q''-\frac{2}{r} Q' + \frac{\sin(2Q)}{r^{2}}=0, \quad Q(1)=0,\; Q'(1)>0
}
with the property that $\lim_{r\to\infty} Q_{n}(r)=n\pi$. Moreover, these $Q_{n}$ are strictly increasing and satisfy 
\EQ{
Q_{n}(r)=n\pi - O(r^{-2})\text{\ \ as\ \ }r\to\infty
}
Now fix any such $Q_{n}$ for $n>0$ and drop the subscript.  Set $\psi(r):= \p_{\lambda} Q(\lambda r)\Big|_{\lambda=1}=r Q'(r)$.
Then $\psi(r)>0$ for all $r\ge1$ and $\psi(r)=O(r^{-2})$ as $r\to\infty$. Furthermore, $\psi$ is a solution to the linearized elliptic problem
\EQ{
-\psi''(r) - \frac{2}{r}\psi'(r) + \frac{2}{r^{2}} \cos(2Q(r))\psi(r)=0
}
in $\R^{3}_{*}$, but it {\em does not} satisfy the Dirichlet condition at $r=1$. 
As before, the $5$-dimensional reduction reads 
\EQ{
\fy(r):=\frac{1}{r}\psi(r),\; (-\Delta_{5}+V) \fy=0,\;  V(r)=\frac{2}{r^{2}}(\cos(2Q(r))-1)
}
where $\Delta_{5}$ is the Laplacian in $\R^{5}$.   
By the preceding, $V$ is a real-valued, radial, bounded and smooth potential on $\R^{5}_{*}$ which decays like $r^{-6}$ as
$r\to\infty$ (and each derivative improves the decay by one power of~$r$). 

The operator $H:= -\Delta+V=-\Delta_{5}+V$ is self-adjoint with domain $\calD:= (H^{2}\cap H^{1}_{0})(\R^{5}_{*})$. 
Its essential spectrum coincides with $[0,\infty)$ and that spectrum is purely absolutely continuous. As observed in~\cite{BCM},
$H$ has no negative spectrum. Indeed, if it did, then by a variational principle there would have to be a lowest eigenvalue $-E^{2}_{*}<0$
which is simple and with associated eigenfunction $f_{*}$ which is smooth, radial, and does not change its sign on $r>1$. 
We may assume that $f_{*}>0$ whence $f_{*}'(1)>0$. 
Then, with $\LR{\cdot|\cdot}$
being the $L^{2}$-pairing in $\R^{5}_{*}$, 
\EQ{
-E^{2}_{*}\LR{f_{*}|\fy} &= \LR{ Hf_{*}| \fy} =  |S^{4}| f_{*}'(1)\fy(1) >0
}
which is a contradiction since the left-hand side is negative.  
It remains to analyze the threshold~$0$, which generally speaking can be either a resonance or an eigenvalue.
Since we are in dimension~$5$, the former would mean that there exists $f\in \calD$, $f\not\equiv 0$, with $|f(x)|\sim \frac{c}{|x|^{3}}$ as $x\to\infty$ 
(the decay here being that of the Newton kernel).
However, in that case $f\in L^{2}$, whence we recover the well-known fact that zero energy can only be an eigenfunction, necessarily radial by our
standing assumption.  Thus, let $Hf=0$, $f\in L^{2}$ radial. Then 
\EQ{\label{Hffy}
0=\LR{Hf|\fy}  = \LR{f|H\fy} + |S^{4}| f'(1)\fy(1)= |S^{4}| f'(1)\fy(1)
}
which is a contradiction since $f(1)=0$ precludes $f'(1)=0$ (recall $\fy(1)\ne0$). In conclusion, $H$ has no point spectrum (as already noted in~\cite{BCM}). 
For future reference we remark that the same argument as in~\eqref{Hffy} shows that  there can be no solution $f\in L^2(\R^5_*)$ of $Hf=0$, unless 
\EQ{\label{vancond}
f'(1)+2f(1)=0
}
Of course $\fy$ satisfies this condition, as can be seen from the equation. 

In order to prove Theorem~\ref{thm:2} we need to establish Strichartz estimates for the wave equation exterior to the ball, perturbed by the radial  potential~$V$.
Once this is done, Theorem~\ref{thm:2} is an immediate consequence via a standard contraction argument. 
Henceforth, the {\em free problem} refers to the wave equation exterior to a ball in $\R^{5}$ with a Dirichlet condition at $r=1$ as considered by~\cite{SmSo}. 
By an {\em admissible Strichartz norm} for the free problem we mean any Strichartz norm as in~\cite{SmSo}  for solutions with $\dot H^{1}_{0}\times L^{2}$-data excluding the
$L^{2}_{t}$-endpoint.

\begin{prop}
\label{prop:Strich V}
Let $\|\cdot\|_{X}$ be an admissible Strichartz norm  for the free problem.
Let $V$ be a potential as above and assume that $-\Delta+V$ has no point spectrum. Then  any solution of 
\EQ{ \label{BoxV}
\Box u + Vu &=F,\quad (t,x)\in (0,\infty)\times\R^{5}_{*}  \\
u(1,t) &=0, \quad t\ge0, \\
(u(0),\dot u(0)) &= (f,g) \in \dot H^{1}_{0}\times L^{2}(\R^{5}_{*})
}
with radial data satisfies
\EQ{\label{Strich X}
\|u\|_{X}\le C\big( \|(f,g)\|_{\dot H^{1}_{0}\times L^{2}} + \|F\|_{L^{1}_{t}L^{2}_{x}}\big)
}
with a constant $C=C(V)$. 
\end{prop}
\begin{proof} The argument is a variant of the one in~\cite{RodS}. 
It suffices to consider $F=0$ by Minkowski's inequality. Let $-\Delta$ be the Laplacian on $\R^{5}_{*}$
with domain $\calD:=H^{2}\cap H^{1}_{0}(\R^{5}_{*})$ on which it is self-adjoint (this incorporates the Dirichlet condition at $r=1$). 
We claim that  $A:=(-\Delta)^{\frac12}$ satisfies
\EQ{\label{ta}
\| A\, f\|_{2} \simeq \|f\|_{\dot H^{1}_{0}}  
}
for all $f\in C^{\infty}(\R^{5})$ which are compactly supported in $\{x\in\R^{5}\mid 1<|x|<\infty\}$.
Indeed, squaring both sides this is equivalent to 
\[
\LR{-\Delta f|f} = \| \nabla f\|_{2}^{2}
\]
for all such $f$, which is obviously true. For any real-valued $u=(u_{1},u_{2})\in \dot H^{1}_{0}\times L^{2}$ we set
\[
U:= A u_{1} + i u_{2}
\]
Then \eqref{ta} implies that $\|U\|_{2}\simeq \| (u_{1},u_{2})\|_{\HH}$. 
Furthermore,  $u$ solves~\eqref{BoxV} if and only if 
\EQ{
i\p_{t} U &= A U + V u \\
U(0) &= A f+ig \in L^{2}(\R^{5}_{*})
}
Then 
\[
U(t) = e^{-it A} U(0) -i \int_{0}^{t} e^{-i(t-s)A} Vu(s)\, ds
\]
By \cite{SmSo},  with $P:= A^{-1}\Re$,  
\[
\| P e^{-it A} U(0)\|_{X}\le C\|U(0)\|_{2}
\]
Factorize $V=V_{1}V_{2}$ where the factors decay like $r^{-3}$. 
By the Christ-Kiselev lemma, see~\cite{SmSo}, and our exclusion of $L^{2}_{t}$, it suffices to bound
\EQ{\label{redux}
\Big \| P \int_{-\infty}^{\infty} e^{-i(t-s) A}\, V_{1}V_{2} \, u(s)\, ds \Big\|_{X} 
&\le \| K\|_{L^{2}_{t,x}\to X} \|V_{2}\, u(s)\|_{L^{2}_{s,x}} \\
(KF)(t) &:= P \int_{-\infty}^{\infty} e^{-i(t-s) A} V_{1}F(s)\, ds
}
Now
\[
\| KF\|_{X}\le \| P e^{-itA} \|_{2\to X} \Big\| \int_{-\infty}^{\infty} e^{is A} V_{1}F(s)\, ds\Big\|_{2}
\]
The first factor on the right-hand side is some constant by~\cite{SmSo}. 
We claim that the second one is bounded by $C\|F\|_{L^{2}_{t,x}}$. By duality, this claim is equivalent to
the {\em local energy bound}
\EQ{
\label{local ener}
\| V_{1}\, e^{-itA} \phi\|_{L^{2}_{t,x}}\le C\|\phi\|_{2}
}
relative to $L^{2}(\R^{5}_{*})$. 
This is elementary to prove for radial $\phi$ (which suffices for us), using the distorted
 Fourier transform relative  to $-\p_{rr}+\frac{2}{r^{2}}$ on $L^{2}((1,\infty))$ with a Dirichlet
condition at $r=1$.  Indeed, map any smooth radial $f=f(r)\in L^{2}(\R^{5}_{*})$ onto the
 function $\tilde f(r)=r^{2}f(r) \in L^{2}(1,\infty)$. Then 
\[
(-\Delta_{5} f )(r) = r^{-2} (\LL_{0} \tilde f)(r),\qquad \LL_{0} = -\p_{rr} + \frac{2}{r^{2}} 
\]
Associated with $\LL_{0}$ there is a distorted Fourier basis $\phi_{0}(r;\lambda)$ that satisfies 
$$\phi_{0}(1;\lambda)=0,\qquad \LL_{0} \phi_{0}(r;\lambda) =\lambda^{2}\phi_{0}(r;\lambda),$$
and such that for all $g\in L^{2}((1,\infty))$
\EQ{\label{calL FT}
\hat{g}(\lambda) &=  \int_{1}^{\infty} \phi_{0}(r;\lambda) g(r)\, dr\\
g(r) &= \int_{0}^{\infty} \phi_{0}(r;\lambda) \hat{g}(\lambda)\, \rho_{0}(d\lambda)\\
\| g\|_{L^{2}(1,\infty)}^{2} &= \int_{0}^{\infty} |\hat{g}(\lambda)|^{2}\, \rho_{0}(d\lambda)
}
where the integrals need to be interpreted in a suitable limiting sense. 
The real-valued functions $\phi_{0}(r;\lambda)$ and the positive measure $\rho_{0}(d\lambda)=\omega_{0}(\lambda) \,d\lambda$ are explicit, see Lemma~\ref{lem:FT} below.
Moreover, it is shown there that 
\EQ{\label{phi0 bd}
\sup_{r\ge1,\:\lambda>0} |\phi_{0}(r;\lambda)|^{2} \omega_{0}(\lambda)\le C<\infty
}
Taking this for granted, we note that~\eqref{local ener} is equivalent to the following estimate for $f\in L^{2}((1,\infty))$
\EQ{\label{LEN}
&\int_{-\infty}^{\infty} \Big\| V_{1}\int_{0}^{\infty}  e^{-it\lambda} \phi_{0}(r;\lambda) \hat{f}(\lambda)\, \rho_{0}(d\lambda)\Big\|_{2}^{2} \, dt 
\le C\|f\|_{2}^{2}
}
Here we used that $A=\sqrt{\LL_{0}}$ (in the half-line picture)
 is given by multiplication by $\lambda$ on the Fourier side, and so $e^{-itA}$ becomes $e^{-it\lambda}$. 
Expanding the left-hand side and carrying out the $t$-integration explicitly reduces this to the following statement:
\EQ{
& \int_{1}^{\infty} V_{1}^{2}(r)   \int_{0}^{\infty}   \int_{0}^{\infty}  \phi_{0}(r;\lambda) \phi_{0}(r;\mu)   \hat{f}(\lambda)\overline{\hat{f}(\mu)}  \de(\lambda-\mu) \,
\rho_{0}( d\lambda )\rho_{0}(d\mu) \, dr
\le C\|f\|_{2}^{2}
}
The left-hand side above is 
\[
= \int_{1}^{\infty} V_{1}^{2}(r)    \int_{0}^{\infty}  \phi_{0}(r;\lambda)^{2} \hat{f}(\lambda)^{2}
\omega_{0}(\lambda)^{2} \, d\lambda \, dr
\]
In view of \eqref{phi0 bd}, \eqref{calL FT}, and 
 $\int_{1}^{\infty} V_{1}^{2}(r)\, dr<\infty$, we obtain~\eqref{LEN}, and thus~\eqref{local ener}. 
This means that $\| K\|_{L^{2}_{t,x}\to X}\le C$, some finite constant. 

For the second factor in~\eqref{redux} we claim the estimate
\EQ{\label{locener2}
 \|V_{2}\,  u(t)\|_{L^{2}_{t,x}} \le C \|U(0)\|_{2} = C\|(f,g)\|_{\dot H^{1}\times L^{2}}
}
valid for any solution of~\eqref{BoxV} with $F=0$. 
To prove it, we invoke the distorted Fourier transform relative to the self-adjoint operator $H:= -\Delta+V$ 
on the domain $\calD$ as defined above, restricted to radial functions. As before, conjugation by $r^{2}$ reduces matters to a half-line operator
 $\LL:= -\p_{rr}+\frac{2}{r^{2}}+V$ on
$L^{2}((1,\infty))$ with a Dirichlet condition at $r=1$. In analogy with $\LL_{0}$, we show in Lemma~\ref{lem:FT} below that there exists 
a Fourier basis $\phi(r;\lambda)$ satisfying for all $\lambda\ge0$
\[
\LL \phi(r;\lambda) = \lambda^{2} \phi(r;\lambda),\quad \phi(1;\lambda)=0
\]
and the correspondences
\EQ{\label{disFT}
\hat{f}(\lambda) &:= \int_{1}^{\infty} \phi(r;\lambda) f(r)\, dr \\
f(r) &= \int_{0}^{\infty}  \phi(r;\lambda)  \hat{f}(\lambda) \, \rho(d\lambda)\\
\|f\|_{L^{2}(1,\infty)} &= \| \hat{f}\|_{L^{2}((0,\infty);\rho)}
}
for a suitable positive measure $\rho(d\lambda)=\omega(\lambda)\, d\lambda$ on $(0,\infty)$.  
It is here that the assumptions on the spectrum of~$H$ enter crucially. Indeed, the absence of negative spectrum means that $\rho$
is supported on $(0,\infty)$, and the absence of a zero eigenvalue implies that $\omega$ exhibits the same rate of decay as~$\omega_{0}$
as~$\lambda\to0+$. The exact property which emerges from all this and which underlies the proof of~\eqref{locener2} is the following variant
of~\eqref{phi0 bd}, see Lemma~\ref{lem:FT}, 
\EQ{\label{phi bd}
\sup_{r\ge1,\:\lambda>0} (\lambda r)^{-2}|\phi(r;\lambda)|^{2} \omega(\lambda)\le C<\infty
} 
The local energy estimate \eqref{locener2} reduces to
\EQ{\nn
&\int_{-\infty}^{\infty} \int_{1}^{\infty} \Big |V_{2}(r) \int_{0}^{\infty} \phi(r;\lambda) \big( \cos(t\lambda) \hat{f}(\lambda) + \lambda^{-1}\sin(t\lambda) \hat{g}(\lambda) \big)\, \rho(d\lambda)    \Big|^{2}\, dr dt \\
&\le C \int_{1}^{\infty} (|f'(r)|^{2}+|g(r)|^{2})\, dr
}
Consider the case $g=0$. Expanding and integrating out the left-hand side one obtains
\EQ{
&\frac12\int_{1}^{\infty} \int_{0}^{\infty} V_{2}(r)^{2}  \phi(r;\lambda)^{2}|\hat{f}(\lambda)|^{2}\, \omega(\lambda)^{2}\, d\lambda\\
&\le C\int_{1}^{\infty} V_{2}(r)^{2} r^{2} \, dr \int_{0}^{\infty}\lambda^{2}|\hat{f}(\lambda)|^{2}\, \rho(d\lambda) \le C \|\sqrt{\LL}f\|_{2}^{2}\le C\|f'\|_{2}^{2}
}
where we used \eqref{phi bd} to pass to the second inequality sign, and~\eqref{ta} to pass to the final inequality.  
The calculation for $f=0$ is similar. 

Putting everything together we obtain \eqref{locener2} and therefore also \eqref{Strich X}. 
\end{proof}

Now we turn to the technical statements concerning the distorted Fourier transforms for the half-line
operators $\LL_{0}=-\p_{rr}+\frac{2}{r^{2}}$ and $\LL=\LL_{0}+V$ on $L^{2}((1,\infty))$, respectively, with
a Dirichlet condition at $r=1$.   This is completely standard, see for example \cite[Section~2]{GZ}, the first two chapters in~\cite{CHS}, or
Newton's survey~\cite{Newton}. But since these references do not treat the specific half-line problem that we are dealing with,
and in order to keep this paper self-contained, we include the details. 

\begin{lem}
\label{lem:FT}
The half-line operators $\LL_{0}$ and $\LL$ admit  Fourier bases satisfying \eqref{calL FT},  \eqref{phi0 bd}, and \eqref{disFT}, \eqref{phi bd},
respectively.  For $\LL$ it is essential to assume that it has no point spectrum. 
\end{lem}
\begin{proof}
For any $z\in\C$ denote by $\phi_{0}(r;z)$ and $\theta_{0}(r;z)$ the unique solutions of
\[
\LL_{0}\phi_{0}(\cdot;z)=z^{2}\phi_{0}(\cdot;z),\quad  \LL_{0}\theta_{0}(\cdot;z)=z^{2}\theta_{0}(\cdot;z)
\]
with initial conditions
\[
\phi_{0}(1;z)=0, \phi_{0}'(1;z)=1,\quad \theta_{0}(1;z)=1, \theta_{0}'(1;z)=0
\]
These are entire in~$z$, and satisfy $W(\theta_0(\cdot;z),\phi_0(\cdot;z))=1$ by construction. Here $W(f,g)=fg'-f'g$ is the Wronskian. 
Furthermore, since $\LL_{0}$ is in the limit-point case at $r=\infty$, for any $z\in\C$ with $\Im z>0$ there exists a unique solution
$\psi_{0}(\cdot;z)\in L^{2}((1,\infty))$  to $\LL_{0} \psi_{0}(\cdot;z)=z^{2}\psi_{0}(\cdot;z)$ with $\psi_{0}(1;z)=1$. Writing
\[
\psi_{0}(\cdot;z)=\theta_{0}(\cdot;z)+m_{0}(z)\phi_{0}(\cdot;z)
\]
one finds that $m_{0}$ is analytic in $\Im z>0$, as well as a Herglotz function ($\Im m(z)>0$ in the upper half plane) 
and the spectral measure is determined by 
\EQ{\label{rhom}
\rho_{0}(d\lambda)=2\lambda\Im m_{0}(\lambda+i0)\,d\lambda
}
It is common to refer to $m_0$ as the {\em Weyl-Titchmarsh function}, and to $\psi$ as the {\em Weyl-Titchmarsh solution}.

For the specific case of~$\LL_{0}$ a fundamental system is of $\LL_0 f =z^2 f$ is given by weighted Hankel functions $r^{\frac12} H^{\pm}_{\frac32}(zr)$. 
These functions are explicit linear combinations of $e^{\pm izr}$ with rational (in~$r$) coefficients. 
Indeed, 
one verifies that
\EQ{\nn
\phi_{0}(r;z) &= (z^{3}r)^{-1} \big[ (1+z^{2}r) \sin(z(r-1)) - z(r-1)\cos(z(r-1)) \big] \\
\theta_0(r;z) &=  (z^{3}r)^{-1} \big[ (1+z^{2}(r-1)) \sin(z(r-1)) +(z^3 r- z(r-1))\cos(z(r-1)) \big] \\
\psi_0(r;z) &= \frac{z+i/r}{z+i} e^{iz(r-1)} \\
m_{0}(z) &= \frac{i(z^2-1)-z}{z+i}
}
Note that 
while the first two lines are entire in~$z$, the third and fourth are meromorphic in~$\C$ and analytic in~$\Im z\ge0$. 
For the spectral measure we find that 
\[
\rho_{0}(d\lambda)=\frac{2\la^4}{1+\la^2}\,d\lambda
\]
To prove \eqref{phi0 bd}, we set $u:=\lambda(r-1)$ whence
\[
\phi_{0}(r;\lambda) = \la^{-2}(u+\la)^{-1}\big[ \sin u -u \cos u + \la(u+\la) \sin u\big] 
\]
If $\la>1$, one checks that $\la \phi_{0}(r;\la)=O(1)$ uniformly in $u>0$, whereas for $0<\la<1$ one has $\la^{2}\phi(r;\la)=O(1)$ for all $u>0$. 
In fact, in both cases one gains a factor of~$u$ for small $u$. 
These two bounds amount to
\[
|\phi_{0}(r;\la)| \frac{\la^{2}}{1+\la}\le C\,\min(1,\la(r-1)) \qquad\forall\; r\ge1, \la>0
\]
which is precisely~\eqref{phi0 bd}. Notice that this estimate contains the $\LL_{0}$-analogue
of~\eqref{phi bd}.

By standard perturbation theory we now transfer these results to~$\LL$, see~\cite{CHS} for more background.
First, for $\la\in\R$, $\la\ne0$,  we set
\EQ{
\label{psiV}
\wt\psi(r;\la) =  \psi_0(r;\la)+\int_r^\infty G_0(r,r';\la)V(r')\wt\psi(r';\la)\, dr'
}
with the Green function 
\[
		G_0(r,r';\la) := \frac{\psi_0(r;\la)\ol{\psi_0(r';\la)}-\psi_0(r';\la)\ol{\psi_0(r;\la)}}{W(\psi_0(\cdot;\la),\ol{\psi_0(\cdot;\la)})}
\]
Evaluating at $r=\infty$ one sees that $W(\psi_0(\cdot;\la),\ol{\psi_0(\cdot;\la)})=-2i\la^3/(1+\la^2)\ne0$.
 To be specific,
\EQ{
G_0(r,r';\la) = \frac{1}{\la^2}(\f{1}{r'}-\f1r)\cos(\la(r-r')) + \frac{\la^2+\frac{1}{rr'}}{\la^2}\frac{\sin(\la(r'-r))}{\la}
}
whence for all $\la\ne0$ and $1<r<r'<\I$, 
\EQ{\label{G0est}
|G_0(r,r';\la)|\le C_0\big(|\la|^{-1}\chi_{[|\la|>1]}+ (r'-r + (r'-r)^3 )\chi_{[0<|\la|<1]}\big)
}
By   Volterra iteration we see that~\eqref{psiV} has a unique solution $\wt\psi(r;\la)$ even for $\la=0$ which satisfies for all $r\ge1$
\EQ{
|\wt\psi(r;\la) -  \psi_0(r;\la)|\le \exp\Big(C_0  \int_r^\I s^3|V(s)|\, ds  \Big)-1
}
We used here that $\|\psi_0(\cdot;\la)\|_{L^\I(1,\I)}\le1$ for all $\la$. It follows that
\EQ{
\wt\psi(r;\la) = \psi_0(r;\la)+O(r^{-4}) \qquad r\to\I
}
uniformly in $\la$. In particular, we conclude that 
\EQ{\label{WbarW}
W\big (\wt\psi(\cdot;\la), \ol{\wt\psi(\cdot;\la)} \big)  = W(\psi_0(\cdot;\la), \ol{\psi_0(\cdot;\la)})=-\frac{2i\la^3}{1+\la^2}
}
whence $\wt\psi(r,\la)\ne0$ for all $\la\ne0$ and $r\ge1$. Hence, we can find a (smooth) function $c(\la)$ for $\la\ne0$
such that 
$
\psi(r;\lambda):=c(\la)\wt\psi(r;\la) 
$ satisfies $\psi(1;\la)=1$. 
Furthermore,  the first estimate in~\eqref{G0est} implies that
\EQ{
\wt\psi(r;\la) = \psi_0(r;\la)+O(\la^{-1}) \qquad \la\to\I
}
uniformly in $r\ge1$.  This shows that $c(\la)=1+O(\la^{-1})$ as $\la\to\I$ and that
\[
2i \Im m(\la)=W(\ol{\psi(\cdot;\la)}, \psi(\cdot;\la)) = \frac{2i\la^3}{1+\la^2} + O(1) \qquad \la\to\I
\]
where $m$ is the Weyl-Titchmarsh function for $\LL$.  In view of the universal property~\eqref{rhom} one has
for all $0<\la_0<\la<\I$
\EQ{
\label{rho1}
C^{-1}\le\la^{-1}\frac{d\rho}{d\la}(\la) \le C
}
for some constant $C=C(\lambda_0)$.   As far as the bounds on $\phi(r;\la)$ are concerned, one has 
\EQ{\label{phimpsi}
\phi(r;\la)=\frac{\Im \psi(r;\la)}{\Im m(\la)} 
}
which immediately shows that for $\la>\la_0$, 
\[
\la |\phi(r;\la)|\le C
\]
To gain a factor $\la(r-1)$, observe that \eqref{psiV} implies that $\|\p_r \psi(r;\la)\|_\I\le C(\la_0)\la$. In particular, 
\[
|\Im \psi(r;\la)|\le |\psi(r;\la)-\psi(1;\la)| \le  C\la(r-1)
\]
where $C=C(\la_0)$ as before.  
It remains to verify~\eqref{disFT}, \eqref{phi bd} in the regime $0<\la\ll1$.  It is of course here that the assumption
on absence of a zero energy eigenvalue enters. 

We begin with the zero energy solution, i.e., a fundamental system of solutions to~$\LL f=0$. 
First, $\frac{1}{r}, r^2$ form such a system for~$\LL_0 f=0$.  Then
\EQ{
\label{u0}
u_0(r) = r^{-1} - \int_r^\I G_0(r,s)V(s) u_0(s)\, ds
}
with Green function 
\[
G_0(r,s) := \frac13\frac{ r^3 - s^3 }{sr}
\]
defines a solution of $\LL u_0=0$. The Volterra iteration again converges and yields 
\EQ{\label{u0as}
u_0(r) = r^{-1}(1+ O(r^{-4})) \qquad r\to\I
}
Here and in what follows, the $O(\cdot)$-terms can be differentiated in $r$ (and $\la$ where appropriate) with
the expected effect. We leave the detailed verification of this property to the reader. 
By \eqref{vancond}, both $u_0(1)\ne0$ and $u_0'(1)\ne0$. Another solution is given by
\EQ{\label{u1}
u_1(r) = u_0(r)\int_{r_0}^r u_0^{-2}(s)\, ds}
for all $r>r_0$ where $r_0\gg1$ is chosen such that $u_0(r)>0$ in that range. 
Inserting~\eqref{u0as} into~\eqref{u1} yields
\EQ{\label{u1as}
u_1(r)=\frac13 r^2 (1+O(r^{-4})) \qquad r\to\I
}
Clearly, $\{u_0,u_1\}$ forms a fundamental system of $\LL u=0$ with $W(u_0,u_1)=1$. 

Next, define for all $r\ge1$ and $0<\la\ll1$, 
\EQ{\label{u1la}
u_1(r;\la) = u_1(r) + \la^2 \int_1^r G(r,r') u_1(r';\la)\, dr'
}
where
\[
G(r,r') := u_1(r) u_0(r') -  u_0(r) u_1(r')
\]
Then \eqref{u1la} has a solution, which satisfies $\LL u_1(\cdot;\la)=\la^2 u_1(\cdot;\la)$ and 
\[
u_1(r;\la) = u_1(r)+O(\la^2 r^2(r-1)^2)
\]
as long as $\la^2 r^2\ll1$.   Similarly, we define $u_0(r;\la)$ as
\EQ{
\label{u0la}
u_0(r;\la) = u_0(r) +\la^2 \Big(  \int_1^r  u_0(r) u_1(s) u_0(s;\la)\, ds + \int_r^{\eps\la^{-1}} u_1(r) u_0(s) u_0(s;\la)\, ds    \Big)
}
Here $\eps>0$ is a small absolute constant, which is to be determined. 
Notice that~\eqref{u0la} is not a Volterra equation, but it can be solved by a contraction argument. Indeed, we 
set 
\[
u_0(r;\la) = u_0(r) + \la^2 r u_2(r;\la)
\]
and reformulate~\eqref{u0la} in the form $u_2 =Tu_2$ for some linear map $T=T_{\eps,\la}$. Then one checks that for all $0<\la\ll1$ and 
a small but fixed $\eps>0$, the map 
$T$ is
a contraction in a ball of fixed size in the space~$C([1,\eps\la^{-1}])$. Consequently,  
there is a unique solution satisfying
\[
|u_2(r;\la)|\le C \qquad \forall\; 1\le r \le \eps\la^{-1}
\]
and all $0<\la\ll1$. Returning to~\eqref{u0la}, we see that this integral equation has 
 a solution for all $1\le r\le\eps\la^{-1}$, which is also a solution of $\LL u_0=\la^2 u_0$, and which  is of the form
\[
u_0(r;\la) =  u_0(r) + O(\la^2 r) \text{\ \ on\ \ }[1,\eps\la^{-1}]
\]
Furthermore,  $\{u_0(\cdot;\la), u_1(\cdot;\la)\}$ forms a fundamental system of $\LL u=\la^2 u$
with $$W(   u_0(\cdot;\la), u_1(\cdot;\la)) =1 + O(\la^2)$$ as $\la\to0$, and $u_0(1;\la)\ne0$ for small~$\la$ since $u_0(1)\ne0$. 

Consequently,  for all $|\la|\ll1$ one has (since $u_1(1;\la)=u_1(1)$)
\EQ{
\phi(r;\la) = c(\la)\Big( u_1(r;\la)-\frac{u_0(r;\la)}{u_0(1;\la)} u_1(1)\Big)
}
where $c(\la)$ is continuous with $|c(\la)|\simeq 1$. Indeed, 
\[
c(\la)= \Big( u_1'(1;\la)-\frac{u_0'(1;\la)}{u_0(1;\la)} u_1(1)\Big)^{-1} = \frac{u_0(1;\la)}{W(   u_0(\cdot;\la), u_1(\cdot;\la))}
\]
 By inspection, one has the bounds  on $1<r<\la^{-1}$, 
 \[
 |\phi(r;\la)|\le C\la^{-2},\qquad |\p_r \phi(r;\la)|\le C\la^{-1}
 \]
 Indeed, $u_1$ satisfies these bounds, and $u_0$ better ones as can be seen directly from the Volterra equations~\eqref{u0la}, \eqref{u1la}. 
Hence, 
\EQ{\label{phiest1}
\la^2 |\phi(r;\la)|\le C\min(1,\la (r-1)) \qquad \forall\; 1<r<\la^{-1}
}
as desired. To extend this bound to $r>\la^{-1}$, and in order to describe the spectral measure for small~$\la$, 
we use~$\tilde\psi$ from~\eqref{psiV}.  In fact, writing
\EQ{
\label{phipsi}
\phi(r;\la) = a(\la) \wt\psi(r;\la) + \bar a(\la) \ol{\wt\psi(r;\la)}
}
one has 
\EQ{\label{ala}
a(\la) = \frac{W(\phi(\cdot;\la), \ol{\wt\psi(\cdot;\la)} )}{ W( \wt\psi(\cdot;\la), \ol{\wt\psi(\cdot;\la)} )} = O(\la^{-3})
}
For the denominator we used \eqref{WbarW}, whereas the numerator is evaluated at $r=\la^{-\frac12}$, say which reduces
matters to  
\EQ{
W\big(\phi(\cdot;\la), \ol{\wt\psi(\cdot;\la)} \big) &= c(\la) W\big(u_1(\cdot;\la), \ol{\psi_0(\cdot;\la)}\big) + o(1) = O(1) \qquad\la\to0
}
Inserting \eqref{ala} into~\eqref{phipsi} one obtains $\sup_{r>1} \la^2 |\phi(r;\la)|=O(1)$ as $\la\to0$. Together with~\eqref{phiest1},
this concludes the proof of~\eqref{phi bd}.

Finally, in order to determine $\Im m(\la)$ for small $\la$, we use the relation~\eqref{phimpsi}, valid for
all $r\ge1$. We use it at $r=C$ a large constant to conclude that 
\[
\phi(r;\la) \asymp 1, \quad \Im \psi(r;\la) \asymp \Im \psi_0(r;\la) \asymp \la^3
\]
which implies $\Im m(\la)\asymp \la^3$ and we are done.  Here $a\asymp b$ means $C^{-1}<\frac{a}{b}<C$. 
\end{proof}

\end{document}